\documentclass[11pt]{article}
\usepackage{amsfonts}

\usepackage{graphics}
\usepackage{indentfirst}
\usepackage{cite}
\usepackage{latexsym}
\usepackage{amsmath,amsthm}
\usepackage{amssymb}
\usepackage{amscd}
\usepackage{mathrsfs}

\usepackage{color}

\newtheorem{theorem}{Theorem}[section]
\newtheorem{remark}{Remark}[section]

\newtheorem{lemma}[theorem]{Lemma}

\newcommand{\n}{\rho}

\renewcommand{\div}{ {\rm div }  }

\newcommand{\na}{\nabla }

\newcommand{\bt}{\begin{theorem}}
\newcommand{\bl}{\begin{lemma}}
\newcommand{\el}{\end{lemma}}
\newcommand{\et}{\end{theorem}}
\newcommand{\ga}{\gamma}

\newcommand{\OM}{\Omega}

\newcommand{\de}{\delta}

\newcommand{\la}{\label}
\newcommand{\si}{\sigma}

\newcommand{\bn}{\begin{eqnarray}}
\newcommand{\en}{\end{eqnarray}}
\newcommand{\bnn}{\begin{eqnarray*}}
\newcommand{\enn}{\end{eqnarray*}}

\newcommand{\bnnn}{\begin{eqnarray*}}
\newcommand{\ennn}{\end{eqnarray*}}
\newcommand{\ben}{\begin{enumerate}}
\newcommand{\een}{\end{enumerate}}

\newcommand{\ba}{\begin{aligned}}
\newcommand{\ea}{\end{aligned}}
\newcommand{\be}{\begin{equation}}
\newcommand{\ee}{\end{equation}}

\def\p{\partial}
\def\norm[#1]#2{\|#2\|_{#1}}

\def\lam{\lambda}
\def\ep{\varepsilon}

\def\rr{\mathbb{R}^2}

\makeatletter      % '@' is now a normal "letter" for TeX
\@addtoreset{equation}{section}
\makeatother       % '@' is restored as a "non-letter" character for TeX

\title{Global Well-Posedness of Classical Solutions to the Multi-Dimensional Degenerate Compressible Navier-Stokes Equations with Large Spherically Symmetric Initial Data}

\date{}

\author{$\text{Qinghao L{\small EI}}^{a,b}\thanks{Email addresses:  leiqinghao22@mails.ucas.ac.cn (Q. H. Lei)}$ \\
a. School of Mathematical Sciences,\\ University of Chinese Academy of Sciences,
Beijing 100049, P. R. China;\\
b. Institute of Applied Mathematics,\\ Academy of Mathematics and Systems Science, \\
Chinese Academy of Sciences, Beijing 100190, P. R. China}

\begin{document}
\maketitle

\begin{abstract}
This paper is concerned with the global existence and uniqueness of classical solutions to the barotropic compressible Navier-Stokes equations with degenerate viscosity coefficients in three-dimensional bounded domains or in the whole space $\mathbb{R}^N$ $(N=2,3)$ with non-vacuum far-field density.
Specifically, we assume that the shear viscosity coefficient $\mu(\rho)=\rho^\alpha$ and the bulk viscosity coefficient $\lambda(\rho)=(\alpha-1)\rho^\alpha$, which satisfy the BD entropy relation.
For arbitrarily large spherically symmetric initial data, we establish the global existence and uniqueness of spherically symmetric classical solutions under the following conditions: for $N=2$, $\alpha \in (0.54369,1)$ and $\gamma \in (1,\infty)$; for $N=3$ (both bounded domains and the whole space), $\alpha \in (0.67661,1)$ and $\gamma \in (1,6\alpha-3)$.
In the two-dimensional case $\mathbb{R}^2$, the restriction on $\alpha$ can be further relaxed to $\alpha \in (9-6\sqrt{2},1)$ provided that the initial data satisfy additional weighted integrability conditions.
Moreover, we show that the solution will not exhibit vacuum in any finite time provided that no vacuum is present initially. 
\end{abstract}

\section{Introduction and main results}
We consider the multi-dimensional barotropic compressible
Navier-Stokes equations with density-dependent viscosity coefficients, which read as follows:
\be\ba\la{ns}
\begin{cases}
\rho_t + \div(\rho \mathbf{u}) = 0, \\
(\n \mathbf{u})_t + \div(\n \mathbf{u} \otimes \mathbf{u}) - \div( \mu(\n) \mathbb{D}\mathbf{u}) - \na( \lambda(\n) \div \mathbf{u} ) + \na P = 0,
\end{cases}
\ea\ee
where $t \ge 0$ is time, $x \in \OM \subset \mathbb{R}^N$ is the spatial coordinate,
$\n=\n(x,t)$ and $\mathbf{u}(x,t)=(u^1(x,t),\dots,u^N(x,t))$ represent the 
density and velocity of the compressible flow, respectively.
The pressure $P$ is given by
\be\la{i1}
P=a \rho^\ga,
\ee
with constants $a>0$ and $\ga > 1$.
Without loss of generality, we assume that $a=1$.
The viscosity coefficients $\mu(\n)$ and $\lam(\n)$ satisfy the physical restrictions
\be\la{i2}
\mu(\n) > 0, \quad \mu(\n) + N \lam(\n) \ge 0,
\ee
and the deformation tensor $\mathbb{D}\mathbf{u}$ is defined by
\be\nonumber
\mathbb{D}\mathbf{u} = \frac{1}{2} ( \na \mathbf{u} + \na \mathbf{u}^T ).
\ee
The system is supplemented with spherically symmetric initial data
\be\la{cz}
\n(x,0)=\n_0(x) = \n_0(r), \quad \mathbf{u}(x,0)=u_0(r) \frac{x}{r}, \quad x\in \OM,
\ee
with $r=|x|$.

In this paper, we mainly investigate two types of boundary conditions:

(1) Cauchy problem: $\OM = \mathbb{R}^N$ $(N =2,3)$ and $(\n,\mathbf{u})$ satisfies the far-field behavior
\be\la{qkjbjtj}\ba
(\n,\mathbf{u})(x,t) \to (\tilde{\n},0) \ \text{ as } |x| \to \infty, \  t>0,
\ea\ee
where $\tilde{\n} > 0$ is a constant.

(2) Three-dimensional bounded domain with Dirichlet boundary condition:
\be\la{yjybjtj}
\OM = B_R \triangleq \{ x \in \mathbb{R}^3 \mid |x| \le R \}, \quad \mathbf{u} = 0 \ \text{ on } \ \p \OM.
\ee

There is a vast body of literature on the global existence of solutions to (\ref{ns}) when both viscosity coefficients $\mu(\n)$ and $\lambda(\n)$ are constants.
The one-dimensional problem has been extensively studied; see \cite{H4,Ka,KN,KS,S1,S2} and the references therein.
For the multi-dimensional case, the local existence and uniqueness of classical solutions were established in \cite{N,S} under the assumption of strictly positive initial density.
For strong solutions, further results were obtained in \cite{CCK,CK,CK2,SS,LLL}, where the initial density is not required to be strictly positive and may vanish on open sets.
The global classical solutions were first obtained by Matsumura-Nishida \cite{MN1} for initial data close to a non-vacuum equilibrium in some Sobolev space $H^s$.
Subsequently, Hoff \cite{H1,H2,H3} investigated the problem with discontinuous initial data and developed a new type of a priori estimates on the material derivative $\dot{\mathbf{u}}$.
For large initial data, a major breakthrough in the global existence theory of weak solutions is due to Lions \cite{L2}, who established the global existence of finite-energy weak solutions with the pressure $P = a\n^\ga \ (a>0,\ga>1)$, provided that $\ga$ is suitably large; for instance, $\ga \ge \frac{9}{5}$ in the three-dimensional case.
This result was later improved by Feireisl-Novotn\'y-Petzeltov\'a \cite{FNP} to $\ga>\frac{3}{2}$, and by Jiang-Zhang \cite{JZ} to $\ga>1$ for spherically symmetric solutions.
For initial data containing vacuum, Huang-Li-Xin \cite{HLX2} proved the global existence and uniqueness of classical solutions to the three-dimensional Cauchy problem for smooth initial data with small energy but possibly large oscillations.
Later, Li-Xin \cite{LX2} further derived some a priori decay rates for both the pressure and the spatial gradient of velocity under the assumption that the initial total energy is sufficiently small, and established the global existence and large-time asymptotic behavior of strong and classical solutions to the Cauchy problem in two and three dimensions with far-fields vacuum.
More recently, Cai-Li \cite{CL} proved the global existence and exponential decay of classical solutions in three-dimensional bounded domains with slip boundary conditions, provided that the initial total energy is suitably small.

It is noteworthy that the case where the viscosity coefficients depend on the density and degenerate at vacuum has received extensive attention in recent years.
Indeed, as pointed out by Liu-Xin-Yang \cite{LXY}, in the derivation of the compressible Navier-Stokes equations from the Boltzmann equation via the Chapman-Enskog expansion (see \cite{CC}), the viscosity depends on the temperature, which translates into a dependence on the density for barotropic flows.
When $\mu$ is a positive constant and $\lambda(\rho) = b \n^\beta$ with $b>0$ and $\beta>3$, Vaigant-Kazhikhov \cite{VK} first proved that the two-dimensional system (\ref{ns}) admits a unique global strong solution for large initial data with density away from vacuum in rectangular domains with slip boundary conditions.
Later, in periodic domains, Jiu-Wang-Xin \cite{JWX1} generalized this result by removing the condition that the initial density be away from vacuum.
More recently, in two-dimensional periodic domains or the whole space, Huang-Li \cite{HL2,HL3} (see also \cite{JWX2}) relaxed the condition $\beta>3$ to $\beta>\frac{4}{3}$.
This problem was further studied by Fan-Li-Li \cite{FLL} in general two-dimensional bounded simply connected domains with Navier-slip boundary conditions, where they established the global existence of strong and weak solutions under the assumption that $\beta>\frac{4}{3}$.
Moreover, Huang-Su-Yan-Yu \cite{HSYY} proved the global existence of radially symmetric strong solutions in two-dimensional balls under the condition of $\beta>1$.

For viscosity coefficients depending on the density and satisfying (\ref{i2}) and
\be\la{bde}
\lambda(\rho) = \mu'(\rho) \rho - \mu(\rho),
\ee
a remarkable framework was introduced by Bresch-Desjardins \cite{BD1,BD2,BDL}, which provides additional information on the gradient of the density.
Subsequently, by deriving a new a priori estimate for smooth approximate solutions, Mellet-Vasseur \cite{MV} studied the stability of (\ref{ns}).
For spherically symmetric initial data, Guo-Jiu-Xin \cite{GJX} established the global existence of weak solutions with large initial data and vacuum.
Furthermore, Guo-Li-Xin \cite{GLX} extended this result to the free boundary problem and investigated the Lagrangian structure and dynamics.
For general initial data, Li-Xin \cite{LX1} and Vasseur-Yu \cite{VY} independently established the global existence of weak solutions with large initial data and vacuum, while Bresch-Vasseur-Yu \cite{BVY} extended these results by considering more general viscosity coefficients.
In the exterior domain of a ball in $\mathbb{R}^{N} \ (N=2,3)$, Cao-Li-Zhu \cite{CLZ} proved the global existence and uniqueness of spherically symmetric classical solutions for large spherically symmetric initial data with far-field vacuum.

More recently, for spherically symmetric initial data, Zhang \cite{Z} first established the global existence and uniqueness of classical solutions in a ball in $\mathbb{R}^{N}$ $(N=2,3)$.
Later, Huang-Meng-Zhang \cite{HMZ} generalized this result and established the large-time asymptotic behavior of solutions.
At the same time, for the Cauchy problem in two and three dimensions, Chen-Zhang-Zhu \cite{CZZ} considered the case $\mu(\n) = \n$ and $\lambda(\n)=0$, which covers the Saint-Venant model for shallow water flows.
They established the global existence and uniqueness of strong and classical solutions for large spherically symmetric initial data with vacuum or non-vacuum far-field density.
Building upon the framework developed in \cite{Z,HMZ}, the main aim of this paper is to establish the global existence and uniqueness of spherically symmetric classical solutions to (\ref{ns})--(\ref{cz}) with the boundary conditions (\ref{qkjbjtj}) or (\ref{yjybjtj}) for large initial data.
Specifically, we consider viscosity coefficients of the form
\be\la{nxxs}
\mu(\n) = \n^\alpha, \quad \lambda(\n) = (\alpha-1) \n^\alpha,
\ee
which satisfy the BD entropy structure (\ref{bde}).

We look for the spherically symmetric solutions of the form:
\be\la{qdc}
\n(x,t) = \n(r,t), \quad \mathbf{u}(x,t) = u(r,t) \frac{x}{r},
\ee
so that system (\ref{ns}) is transformed into
\be\la{nsqdc}\ba
\begin{cases}
\n_t + (\n u)_r + \frac{(N-1)}{r} \n u = 0, \\
\n(u_t + u u_r) + (\n^\ga)_r - \alpha \left( r^{-(N-1)} \n^\alpha (r^{N-1}u)_r \right)_r + \frac{N-1}{r}(\n^\alpha)_r u = 0.
\end{cases}
\ea\ee

Before stating the main results, we first explain the notations
and conventions used throughout this paper.
We denote
\be\ba\nonumber
\int f dx \triangleq \int_{\OM} f dx.
\ea\ee

For any $1\leq s \leq \infty$ and a positive integer $k$, we define the standard Lebesgue and Sobolev spaces as follows:
\be\ba\nonumber
\begin{cases}
L^s =L^s(\OM),\quad W^{k,s} =W^{k,s}(\OM),\quad H^k = W^{k,2}, \\
D^{k,s}=D^{k,s}(\OM)=\{ v\in L^1_\mathrm{loc}(\OM) \mid \nabla ^k v\in L^s(\OM) \}.
\end{cases}
\ea\ee
For the Cauchy problem, we define the potential energy density as
\be\la{pe}\ba
K(\n) \triangleq \n \int_{\tilde{\rho}}^{\n} \frac{P(s)-P(\tilde{\rho})}{s^2} ds.
\ea\ee
We introduce the effective velocity $w$:
\be\la{yxsd}\ba
w \triangleq u + \n^{-1} (\n^\alpha)_r,
\ea\ee
which, by (\ref{nsqdc}), implies that $w$ satisfies
\be\la{yxsdfc1}\ba
\n w_t + \n u w_r + (\n^\ga)_r = 0.
\ea\ee

We now state the first main result concerning the global existence of classical solutions in the whole space.
\begin{theorem}\la{th1}
Let $\OM = \mathbb{R}^N$ $(N=2,3)$.
Assume that $\alpha$ and $\ga$ satisfy
\be\la{th11}\ba
& 0.54369 < \alpha <1, \quad 1<\ga<\infty \quad \ \  \  \textnormal{ if } N=2, \\
& 0.67661 < \alpha<1, \quad 1<\ga<6\alpha-3 \ \textnormal{ if } N=3,
\ea\ee
and the initial data $(\n_0,\mathbf{u}_0)$ are spherically symmetric and satisfy
\be\la{th12}\ba
0 < \underline{\n_0} \le \n_0 \le \overline{\n_0}, \quad
\n_0 - \tilde{\n} \in H^3(\mathbb{R}^N), \quad \mathbf{u}_0 \in H^3(\mathbb{R}^N),
\ea\ee
where $\underline{\n_0}$ and $\overline{\n_0}$ are positive constants.
Then the problem \eqref{ns}--\eqref{qkjbjtj}, \eqref{nxxs} has a unique global spherically symmetric classical solution $(\n,\mathbf{u})$ in $\mathbb{R}^{N} \times (0,\infty)$ satisfying, for any $0 < T <\infty$,
\be\la{th13}\ba
(C(T))^{-1} \le \n(x,t) \le C(T), \quad \textnormal{ for all } (x,t)\in \mathbb{R}^N \times[0,T],
\ea\ee
and
\be\la{th14}\ba
\begin{cases}
\n - \tilde{\n} \in C([0,T];H^3(\mathbb{R}^N)), \n_t \in C([0,T];H^2(\mathbb{R}^N)), \\
\mathbf{u} \in C([0,T];H^3(\mathbb{R}^N)) \cap L^2([0,T];H^4(\mathbb{R}^N)), \\
\mathbf{u}_t \in C([0,T];H^1(\mathbb{R}^N)) \cap L^2([0,T];H^2(\mathbb{R}^N)), \mathbf{u}_{tt} \in L^2(\mathbb{R}^N \times (0,T)),
\\
\sqrt{t} \na^4 \mathbf{u}, \sqrt{t} \na^2 \mathbf{u}_{t}, \sqrt{t} \mathbf{u}_{tt} \in L^\infty((0,T);L^2(\mathbb{R}^N)), \sqrt{t} \na \mathbf{u}_{tt} \in L^2(\mathbb{R}^N \times (0,T)),
\end{cases}
\ea\ee
where $C(T)>0$ is a constant depending only on $T$, $\tilde{\n}$, $\alpha$, $\ga$, $\underline{\n_0}$, $\| \n_0 - \tilde{\n} \|_{H^3}$, and $\| \mathbf{u}_0 \|_{H^2}$.
\end{theorem}

In the two-dimensional case $\mathbb{R}^2$, the restriction on $\alpha$ obtained in Theorem \ref{th1} can be further relaxed if the initial data satisfy additional weighted integrability conditions.
\begin{theorem}\la{th212}
Let $\OM = \mathbb{R}^2$.
Assume that $\alpha$ and $\ga$ satisfy
\be\la{th121}\ba
& 9 - 6 \sqrt{2} < \alpha < 1, \quad 1<\ga<\infty,
\ea\ee
and the initial data $(\n_0,\mathbf{u}_0)$ are spherically symmetric and satisfy \eqref{th12} and for some $\eta \in [\frac{1}{3},1]$,
\be\la{th122}\ba
|x|^{\frac{\eta}{2}} (\n_0 - \tilde{\n}), \ |x|^{\frac{\eta}{2}} \na \n_0, \ |x|^{\frac{\eta}{2}} \mathbf{u}_0 \in L^2(\mathbb{R}^2).
\ea\ee
Then the problem \eqref{ns}--\eqref{qkjbjtj}, \eqref{nxxs} has a unique global spherically symmetric classical solution $(\n,\mathbf{u})$ in $\mathbb{R}^{2} \times (0,\infty)$ satisfying \eqref{th14} and, for any $0 < T <\infty$,
\be\la{th123}\ba
(C(T))^{-1} \le \n(x,t) \le C(T), \quad \textnormal{ for all } (x,t)\in \mathbb{R}^2 \times[0,T],
\ea\ee
and
\be\la{th124}\ba
|x|^{\frac{\eta}{2}} (\n - \tilde{\n}), \ |x|^{\frac{\eta}{2}} \na \n, \ |x|^{\frac{\eta}{2}} \mathbf{u} \in L^\infty(0,T;L^2(\mathbb{R}^2)),
\ea\ee
where $C(T)>0$ is a constant depending only on $\eta$, $T$, $\tilde{\n}$, $\alpha$, $\ga$, $\underline{\n_0}$, $\| \n_0 - \tilde{\n} \|_{H^3}$, $\| \mathbf{u}_0 \|_{H^2}$, $\| |x|^{\frac{\eta}{2}} (\n_0 - \tilde{\n}) \|_{L^2}$, $\| |x|^{\frac{\eta}{2}} \na \n_0 \|_{L^2}$, and $\| |x|^{\frac{\eta}{2}} \mathbf{u}_0 \|_{L^2}$.
\end{theorem}

\begin{remark}\la{rk22}
Since $9-6\sqrt{2} \approx 0.51472$, the restriction on $\alpha$ in Theorem \ref{th1} for the two-dimensional case can be improved by imposing additional weighted integrability conditions on the initial data.
\end{remark}

For the three-dimensional bounded domain, we have the following result on the global existence of classical solutions.
\begin{theorem}\la{th2}
Let $\OM=B_R$.
Assume that $\alpha$ and $\ga$ satisfy
\be\la{th21}\ba
0.67661 < \alpha<1, \quad 1<\ga<6\alpha-3,
\ea\ee
and the initial data $(\n_0,\mathbf{u}_0)$ are spherically symmetric and satisfy
\be\la{th22}\ba
0 < \underline{\n_0} \le \n_0 \le \overline{\n_0}, \quad
\n_0 - \tilde{\n} \in H^3(B_R), \quad \mathbf{u}_0 \in H_0^1(B_R) \cap H^3(B_R),
\ea\ee
where $\underline{\n_0}$ and $\overline{\n_0}$ are positive constants.
Then the problem \eqref{ns}--\eqref{cz}, \eqref{yjybjtj}, \eqref{nxxs} has a unique global spherically symmetric classical solution $(\n,\mathbf{u})$ in $B_R \times (0,\infty)$ satisfying, for any $0 < T <\infty$,
\be\la{th23}\ba
(C(T))^{-1} \le \n(x,t) \le C(T), \quad \textnormal{ for all } (x,t)\in B_R \times[0,T],
\ea\ee
and
\be\la{th24}\ba
\begin{cases}
\n - \tilde{\n} \in C([0,T];H^3(B_R)), \n_t \in C([0,T];H^2(B_R)), \\
\mathbf{u} \in C([0,T];H^3(B_R)) \cap L^2([0,T];H^4(B_R)), \\
\mathbf{u}_t \in C([0,T];H^1(B_R)) \cap L^2([0,T];H^2(B_R)),\mathbf{u}_{tt} \in L^2(B_R \times (0,T)),
\\
\sqrt{t} \na^4 \mathbf{u}, \sqrt{t} \na^2 \mathbf{u}_{t}, \sqrt{t} \mathbf{u}_{tt} \in L^\infty((0,T);L^2(B_R)), \sqrt{t} \na \mathbf{u}_{tt} \in L^2(B_R \times (0,T)),
\end{cases}
\ea\ee
where $C(T)>0$ is a constant depending only on $R$, $T$, $\alpha$, $\ga$, $\underline{\n_0}$, $\| \n_0 - \tilde{\n} \|_{H^3}$, and $\| \mathbf{u}_0 \|_{H^2}$.
\end{theorem}

\begin{remark}\la{lrk1}
It follows from \eqref{th13}, \eqref{th14}, \eqref{th23}, \eqref{th24} and the Sobolev embedding that
\be\la{csol1}\ba
\n, P(\n), \mu(\n), \lambda(\n) \in C\left([0,T];C^1(\OM) \right), \quad \n_t \in C\left([0,T];C(\OM) \right).
\ea\ee
Moreover, the standard embedding together with \eqref{th14} and \eqref{th24} yields that for any $0<\tau<T<\infty$,
\be\la{csol2}\ba
\mathbf{u} \in L^\infty(\tau,T;H^{4}(\OM)) \cap H^1(0,T;H^2(\OM)) \hookrightarrow
C\left([\tau,T];C^2(\OM) \right),
\ea\ee
and	
\be\la{csol3}\ba
\mathbf{u}_t \in L^\infty(\tau,T;H^2(\OM))\cap H^1(\tau ,T;H^1(\OM))\hookrightarrow
C\left([\tau,T]; C(\OM) \right).
\ea\ee
Combining \eqref{csol1}--\eqref{csol3}, we conclude that the solutions obtained in Theorems \ref{th1} and \ref{th2} are in fact classical solutions.
\end{remark}

\begin{remark}\la{rk1}
It should be emphasized that the lower bound restriction for $\alpha$ in \eqref{th11} and \eqref{th21} is stated in an approximate numerical form.
The exact threshold can be characterized as follows.

For $\OM = \mathbb{R}^2$, the exact threshold is
\be\la{k02}\ba
1 - \frac{2}{k_1} < \alpha < 1,
\ea\ee
where
\be\la{k1}\ba
k_1 = 2 + \left( 2 + \sqrt{\frac{44}{27}} \right)^{\frac{1}{3}} + \left( 2 - \sqrt{\frac{44}{27}} \right)^{\frac{1}{3}},
\ea\ee
which is the unique real root in $(2,\infty)$ of
\be\nonumber\ba
k_1^3 - 6k_1^2+8k_1-4 = 0.
\ea\ee
Moreover, $k_1$ satisfies
\be\la{k12}
1 - \frac{2k_1\sqrt{k_1-1}-4k_1+4}{(k_1-2)^2} = 1 - \frac{2}{k_1},
\ee
and for any $k>k_1$,
\be\la{k13}
1 - \frac{2k\sqrt{k-1}-4k+4}{(k-2)^2} < 1 - \frac{2}{k}.
\ee
For $\OM = \mathbb{R}^3$ or $\OM = B_R$, the exact threshold is given by
\be\la{k03}\ba
1 - \frac{1}{k_2} < \alpha < 1,
\ea\ee
where
\be\la{k2}\ba
k_2 = \frac{3}{2} + \left( \frac{5}{8} + \sqrt{\frac{83}{432}} \right)^{\frac{1}{3}} + \left( \frac{5}{8} - \sqrt{\frac{83}{432}} \right)^{\frac{1}{3}},
\ea\ee
which is the unique real root in $(2,\infty)$ of
\be\nonumber\ba
2 k_2^3 - 9 k_2^2 + 10 k_2 - 4 = 0.
\ea\ee
Furthermore, $k_2$ satisfies
\be\la{k21}
1 - \frac{\sqrt{2k_2^3-k_2^2-2k_2+1}-3k_2+3}{(k_2-2)^2} = 1 - \frac{1}{k_2},
\ee
and for any $k>k_2$,
\be\la{k22}
1 - \frac{\sqrt{2k^3-k^2-2k+1}-3k+3}{(k-2)^2} < 1 - \frac{1}{k}.
\ee
\end{remark}

\begin{remark}\la{rk3}
In \cite[Theorem 2.1]{HMZ}, the authors established the global existence of classical solutions in a three-dimensional bounded domain under the following conditions:
\be\nonumber
0.686<\alpha<1, \quad 1<\gamma<6\alpha-3+\frac{3-5\alpha}{2n_3(\alpha)},
\ee
where $n_3(\alpha)$ is defined as follows:
Define
\begin{equation}\nonumber
\alpha_{3,-}(n) = 1 - \frac{\sqrt{4n(4n^2-n-1)+1}-6n+3}{4n^2-8n+4}.
\end{equation}
It is strictly increasing on $(1,\infty)$ and satisfies
\begin{equation}\nonumber
\lim_{n \to 1^+} \alpha_{3,-}(n) = \frac{2}{3}, \quad \lim_{n \to \infty} \alpha_{3,-}(n) = 1.
\end{equation}
Let $n_3(\alpha)$ denote the inverse function of $\alpha_{3,-}(n)$, mapping $(\frac{2}{3},1)$ to $(1,\infty)$.
Then $n_3(\alpha)$ is strictly increasing on $(\frac{2}{3},1)$ and satisfies
\begin{equation}\nonumber
\lim_{\alpha \to \frac{2}{3}^+} n_3(\alpha) = 1, \quad \lim_{\alpha \to 1^-} n_3(\alpha) = \infty.
\end{equation}
Therefore, Theorem \ref{th2} generalizes the corresponding result in \cite[Theorem 2.1]{HMZ} for the three-dimensional case.
\end{remark}

\begin{remark}\la{lrk5}
It is worth noting that our results hold in the two-dimensional case for any $1<\ga<\infty$, while in the three-dimensional case we require $1<\ga< 6\alpha-3$.
The essential difference lies in the geometric factor $r^{N-1}$ arising from the spherical coordinate transformation.
More precisely, the BD entropy estimate provides an $L^\infty(0,T;L^2(\mathbb{R}^N))$ estimate for $\na \rho^{\alpha-\frac{1}{2}}$.
Through the spherical coordinate transformation, this yields an $L^\infty(0,T;L^2(0,\infty))$ bound for $(\p_r \rho^{\alpha-\frac{1}{2}}) r^{\frac{N-1}{2}}$.
Due to the singularity near the origin, we cannot directly obtain an $L^\infty(0,1)$ bound for the density.
In fact, by applying the one-dimensional Sobolev inequality, we can establish the following $r$-weighted estimates: for any $\xi>0$,
\be\nonumber\ba
\sup_{0 \le t \le T} \| \n r^\xi \|_{L^\infty(0,1)} \le C(\xi) \quad & \textnormal{ if } \  N=2, \\
\sup_{0 \le t \le T} \| \n r^{\frac{1+\xi}{2\alpha-1}} \|_{L^\infty(0,1)} \le C(\xi) \quad & \textnormal{ if } \  N=3.
\ea\ee
Note that in the two-dimensional case, the weighted exponent can be taken arbitrarily small, whereas in the three-dimensional case, it can no longer be taken arbitrarily small.
These weighted estimates are mainly used to bound the pressure term $\n^\ga$.
This distinction in the weighted exponents is precisely why the upper bound restriction on $\ga$ is needed when $N=3$.
\end{remark}

We now make some comments on the analysis of this paper.
Note that for initial data satisfying (\ref{th12}), the local existence and uniqueness of classical solutions to the problem \eqref{ns}--\eqref{cz}, \eqref{nxxs} with the boundary conditions \eqref{qkjbjtj} or \eqref{yjybjtj} can be established by arguments similar to those in \cite{CK,LPZ,ZZ,Z}.
Hence, to extend the local classical solution globally in time, it is essential to derive global a priori estimates in suitable higher-order norms.
The key point is to obtain the upper and lower bounds of the density.

For the two-dimensional case, we first combine the standard energy estimate with the BD entropy estimate and the Sobolev inequality to obtain
\be\la{cm1}
\sup_{0 \le t \le T} \| (\n-\tilde{\n}) r^{\frac{1}{p}} \|_{L^p(0,\infty)} \le C(p), \quad \text{ for all} \  2 \le p <\infty,
\ee
and
\be\la{cm2}
\sup_{0 \le t \le T} \| \p_r \n^{\alpha-\frac{1}{2}} r^{\frac{1}{2}} \|_{L^2(0,\infty)} \le C.
\ee
Combining these estimates with the Sobolev inequality and taking advantage of the spherical symmetry, we can derive
\be\la{cm3}
\sup_{0 \le t \le T} \| \n \|_{L^\infty(1,\infty)} \le C.
\ee
However, due to the singularity at the origin, (\ref{cm1}) and (\ref{cm2}) do not directly yield an upper bound for the density near the origin.
Indeed, we can only obtain an $r$-weighted estimate:
\be\nonumber
\sup_{0 \le t \le T} \| \n r^\xi \|_{L^\infty(0,1)} \le C(\xi), \quad \text{ for all} \  \xi>0.
\ee
Thus, to derive the $L^\infty(0,T;L^\infty(0,1))$ bound for $\n$, it is necessary to establish higher integrability of $\p_r \n$.
This can be achieved via higher integrability of $\n u$ and $\n w$.
To this end, we derive from $(\ref{nsqdc})_2$ and (\ref{yxsdfc1}) the $L^\infty(0,T;L^k(0,\infty))$ estimates for $(\n r)^{\frac{1}{k}} u$ and $(\n r)^{\frac{1}{k}} w$ with $k>2$.
A careful analysis shows that $k$ must satisfy
\be\nonumber
1 - \frac{2k\sqrt{k-1}-4k+4}{(k-2)^2} < \alpha.
\ee
Using these estimates and the definition of the effective velocity $w$, we further obtain the $L^\infty(0,T;L^k(0,\infty))$ bound for $(\n^{k (\alpha-2)+1} r)^{\frac{1}{k}} \p_r \n$.
Then, by virtue of the one-dimensional Sobolev inequality, we arrive at
\be\nonumber
\sup_{0 \le t \le T} \| \n \|_{L^\infty(0,1)} \le C,
\ee
which together with (\ref{cm3}) yields the desired upper bound of $\n$.

Next, we turn to the crucial lower bound estimate for $\n$.
Note that in the mass (Lagrangian) coordinates, the integral of $\n^{-1}$ represents the volume of the fluid, which is conserved in bounded domains.
Therefore, we introduce Lagrangian coordinates to derive the lower bound for $\n$ near the origin.
We define the coordinate transformation, for any $r \in (0,\infty)$,
\be\la{lzb1}
y(r,t) = \int_0^r \rho(s,t) s^{N-1} \, ds, \quad \tau(r,t) = t.
\ee
This, together with (\ref{nsqdc}), implies
\be\la{lzb2}
\frac{\partial y}{\partial r} = \rho r^{N-1}, \quad \frac{\partial y}{\partial t} = -\rho u r^{N-1}, \quad \frac{\partial \tau}{\partial r} = 0, \quad \frac{\partial \tau}{\partial t} = 1, \quad \frac{\partial r}{\partial \tau} = u.
\ee
Moreover, for any $r \in (0,\infty)$ with $y$ defined in (\ref{lzb1}), we have
\be\nonumber
\int_0^y \n^{-1} dy = \int_0^r \n^{-1} \n r^{N-1} dr = \frac{1}{N} r^N.
\ee
Combining this with the $L^\infty(0,T;L^k(0,\infty))$ estimate for $(\n^{k (\alpha-2)+1} r)^{\frac{1}{k}} \p_r \n$, we can obtain the $L^\infty(0,1)$ bound for $\n^{-1}$.
Since the domain is unbounded, this argument cannot be directly applied to derive the $L^\infty(1,\infty)$ bound for $\n^{-1}$.
We therefore work in Eulerian coordinates.
For the transport equation $(\ref{nsqdc})_1$, we employ a standard energy-type estimate to obtain the $L^\infty(1,\infty)$ bound for $\n^{-1}$, provided that there exists some $k \ge 3$ such that
\be\nonumber
1 - \frac{2k\sqrt{k-1}-4k+4}{(k-2)^2} < \alpha < 1 - \frac{2}{k}.
\ee
For such a $k$ to exist, we need $\alpha$ to satisfy (\ref{k02}).

If the initial data further satisfy (\ref{th122}), we derive from $(\ref{nsqdc})_2$ and $(\ref{yxsdfc1})$ the $L^\infty(0,T;L^2(0,\infty))$ bounds for $\sqrt{\n} u r^{\frac{1+\eta}{2}}$ and $\sqrt{\n} w r^{\frac{1+\eta}{2}}$.
These estimates, in turn, imply the $L^\infty(0,T;L^2(0,\infty))$ bounds for $\p_r \n^{\alpha-\frac{1}{2}} r^{\frac{1+\eta}{2}}$.
Using these bounds and the previous estimates, we obtain the upper and lower bounds for the density under the condition that there exists some $k \ge 3$ such that
\be\nonumber
1 - \frac{2k\sqrt{k-1}-4k+4}{(k-2)^2} < \alpha < 1 - \frac{1}{k}.
\ee
Such a $k$ exists provided that $\alpha$ satisfies (\ref{th121}).

For the three-dimensional case, following arguments similar to those in the two-dimensional case, we can derive the standard energy estimate, the BD entropy estimate, and the $L^\infty(0,T;L^k(0,\infty))$ estimates for $(\n r^2)^{\frac{1}{k}} u$ and $(\n r^2)^{\frac{1}{k}} w$ with $k \ge 3$.
These estimates yield the upper bound for the density, provided that there exists $k$ such that
\be\nonumber
1 - \frac{\sqrt{2k^3-k^2-2k+1}-3k+3}{(k-2)^2} < \alpha.
\ee
We now establish the lower bound of $\n$.
On the one hand, in Lagrangian coordinates and following arguments similar to the two-dimensional case, we can obtain the $L^\infty(0,1)$ bound for $\n^{-1}$.
On the other hand, the BD entropy estimate combined with the Cauchy inequality implies
\be\nonumber\ba
\int_1^\infty |\p_r \n^{\alpha-1}| dr \le \frac{1}{2} \int_1^\infty \left( |\p_r \n^{\alpha-1}|^2 r^2 + r^{-2} \right) dr \le C.
\ea\ee
Combining this with the above estimates yields the $L^\infty(1,\infty)$ bound for $\n^{-1}$.
In the derivation, we need to find $k > 3$ such that
\be\nonumber
1 - \frac{\sqrt{2k^3-k^2-2k+1}-3k+3}{(k-2)^2} < \alpha < 1 - \frac{1}{k}.
\ee
For such a $k$ to exist, $\alpha$ is required to satisfy (\ref{k03}).

With these upper and lower bounds of the density at hand, we adapt the standard arguments in \cite{CL,HLX2,HMZ} to establish higher-order estimates, which ensure that the local solution can be extended globally in time.

The rest of this paper is organized as follows:
Section 2 introduces some essential inequalities and known facts.
Sections 3 and 4 are devoted to deriving the upper bound and lower bound estimates of the density for $\OM=\mathbb{R}^2$ and $\OM=\mathbb{R}^3$, respectively.
In Section 5, we establish the higher-order estimates.
Finally, the proofs of Theorems \ref{th1} and \ref{th2} are presented in Section 6.

\section{Preliminaries}
In this section, we recall some known results and elementary inequalities
that will be used frequently in subsequent analysis.

First, we recall the following local existence result of classical solutions, which can be proved by arguments similar to those in \cite{CK,LPZ,ZZ,Z}.
\begin{lemma}\la{lct}
Assume that the initial data $(\n_0, \mathbf{u}_0)$ satisfy \eqref{th12} or \eqref{th22}.
Then there is a small time $T>0$ depending only on $\alpha$, $\ga$, $\tilde{\n}$, and the initial data such that the problem \eqref{ns}--\eqref{cz}, \eqref{nxxs} with the boundary conditions \eqref{qkjbjtj} or \eqref{yjybjtj} admits a unique classical solution $(\n,\mathbf{u})$ on $\OM \times (0,T]$ satisfying \eqref{th13}, \eqref{th14} or \eqref{th23}, \eqref{th24}.
Moreover, if the initial data further satisfy \eqref{th122}, then the local classical solution satisfies \eqref{th124}.
\end{lemma}

The following lemma shows that if the initial data is spherically symmetric, then the local classical solution of (\ref{ns})--(\ref{cz}), (\ref{nxxs}) is also spherically symmetric and the velocity field vanishes at the origin.
\begin{lemma}\la{dcx}
Assume that the initial data are spherically symmetric.
Then the local classical solution $(\rho,\mathbf{u})$ of \eqref{ns}--\eqref{cz}, \eqref{nxxs} is also spherically symmetric.
Moreover, the velocity field vanishes at the origin, i.e., $\mathbf{u}(0)=0$, and the radial profile $u$ can be continuously extended to $r=0$ by setting $u(0)=0$.
\end{lemma}
\begin{proof}
Let $Q \in SO(N)$ be an arbitrary $N \times N$ orthogonal matrix.
Define
\be\la{}\ba
\rho_{Q}(x,t) \triangleq \rho(Qx,t), \quad 
\mathbf{u}_{Q}(x,t) \triangleq Q^{T} \mathbf{u}(Qx,t).
\ea\ee
Since the system \eqref{ns}--\eqref{cz}, \eqref{nxxs} is rotationally invariant and the initial data $(\rho_0,\mathbf{u}_0)$ are spherically symmetric, it is straightforward to verify that $(\rho_{Q},\mathbf{u}_{Q})$ is also a solution to \eqref{ns}--\eqref{cz}, \eqref{nxxs} with the initial data $(\rho_0,\mathbf{u}_0)$.
By the uniqueness of classical solutions to \eqref{ns}--\eqref{cz}, \eqref{nxxs}, we obtain
\be\nonumber
\rho_{Q} = \rho, \quad \mathbf{u}_{Q} = \mathbf{u}
\ee
Thus, the local classical solution $(\rho,\mathbf{u})$ is spherically symmetric.

Next, we show that $\mathbf{u}(0)=0$.
Let $e_1$ and $e_2$ be two distinct constant unit vectors.
Since $\mathbf{u}$ is spherically symmetric, we have for any $r>0$,
\be\nonumber\ba
\mathbf{u}(re_1) - \mathbf{u}(re_2) = u(r)e_1 - u(r)e_2 = u(r)(e_1 - e_2).
\ea\ee
By the continuity of $\mathbf{u}$ at the origin, we arrive at
\be\nonumber\ba
0 = \lim_{r \to 0^+} \left( \mathbf{u}(re_1) - \mathbf{u}(re_2) \right)
= \lim_{r \to 0^+} u(r)(e_1 - e_2),
\ea\ee
which together with the fact that $e_1 \ne e_2$ yields
\be\la{dcxp1}\ba
\lim_{r \to 0^+} u(r) = 0.
\ea\ee
This implies that
\be\nonumber\ba
\mathbf{u}(0) = \lim_{r \to 0^+} \mathbf{u}(r e_1) = \lim_{r \to 0^+} u(r) e_1 = 0.
\ea\ee
Finally, (\ref{dcxp1}) ensures that $u$ can be continuously extended to $r=0$ by setting $u(0)=0$.
\end{proof}

Next, we will frequently use the following Gagliardo-Nirenberg inequalities (see \cite{NI}).
\begin{lemma}\la{gn1}
Let $p, q, r \in [1,\infty]$ and $N, j, m \in \mathbb{N}$ satisfy
\be\nonumber\ba
\frac{1}{p} = \frac{j}{N} + a \left( \frac{1}{r} - \frac{m}{N} \right) + \frac{1-a}{q},
\ea\ee
where $a \in [\frac{j}{m},1]$.
In addition, if $m-j-\frac{N}{r}=0$, then it is necessary to also assume that $a<1$.
Then for $f \in L^q(\mathbb{R}^N) \cap D^{m,r}(\mathbb{R}^N)$, there exists a positive constant $C$ depending only on
$p$, $q$, $r$, $N$, $j$, $m$, and $a$ such that
\be\ba\la{gn11}
\| \na^j f \|_{L^p} \le C \| f \|^{1-a}_{L^q} \| \na^m f\|^{a}_{L^r}.
\ea\ee
\end{lemma}

The following one-dimensional Sobolev inequality will be frequently used.
\begin{lemma}\la{ywsi}
Let $f \in W^{1,1}(0,\infty)$.
Then for any $0 \le r < \infty$,
\be\la{ywsi1}\ba
|f(r)| \le \int_{r}^\infty |f'(s)| ds.
\ea\ee
Moreover, for any $0 \le r_1 < r < r_2 < \infty$,
\be\la{ywsi2}\ba
|f(r)| \le \frac{1}{r_2-r_1} \int_{r_1}^{r_2} |f(s)| ds
+ \int_{r_1}^{r_2} |f'(s)| ds.
\ea\ee
\end{lemma}
\begin{proof}
First, since $f \in W^{1,1}(0,\infty)$, it follows from \cite{AF,EL} that $f \in C([0,\infty))$ and
\be\la{ywp1}\ba
f(r) = f(R) - \int_r^R f'(s) ds \quad \text{ for all } 0 \le r < R < \infty .
\ea\ee
This implies that, for any $0 \le R_1<R_2<\infty$,
\be\la{ywp2}\ba
|f(R_1) - f(R_2)| = \left| \int_{R_1}^{R_2} f'(s) ds \right|
\le \int_{R_1}^{R_2} |f'(s)| ds.
\ea\ee
Since $f' \in L^1(0,\infty)$, the right-hand side of (\ref{ywp2}) vanishes as $R_1, R_2 \to \infty$, and hence $\lim_{R \to \infty} f(R)$ exists.
Combining this with $f \in L^1(0,\infty)$ leads to
\be\la{ywp3}\ba
\lim_{R \to \infty} f(R) = 0.
\ea\ee
Letting $R \to \infty$ in (\ref{ywp1}), we obtain
\be\la{ywp4}\ba
f(r) = - \int_r^\infty f'(s) ds,
\ea\ee
which yields (\ref{ywsi1}).

Next, for any $r, \hat{r} \in (r_1,r_2)$, we deduce from (\ref{ywp1}) that
\be\la{ywp6}\ba
|f(r)| = \left| f(\hat{r}) - \int_{r}^{\hat{r}} f'(s) ds \right|
\le |f(\hat{r})| + \int_{r_1}^{r_2} |f'(s)| ds.
\ea\ee
Integrating (\ref{ywp6}) with respect to $\hat{r}$ over $(r_1,r_2)$ and dividing by $r_2-r_1$, we arrive at (\ref{ywsi2}).
This completes the proof of Lemma \ref{ywsi}.
\end{proof}

\section{Upper Bound and Lower Bound of the Density ($N=2$)}
In this section, we derive both the upper bound and the positive lower bound of the density for the case $\OM=\mathbb{R}^2$.
Assume that the initial data $(\n_0,\mathbf{u}_0)$ satisfy (\ref{th12}).
Let $(\n,\mathbf{u})$ be a spherically symmetric classical solution to (\ref{ns})--(\ref{qkjbjtj}), (\ref{nxxs}) on $\mathbb{R}^2 \times (0,T]$, whose existence is guaranteed by Lemmas \ref{lct} and \ref{dcx}.

First, we derive the standard energy estimate and the BD entropy estimate.
\begin{lemma}\la{2l1}
For any $p \in [2,\infty)$, there exists a positive constant $C$ depending only on
$T$, $p$, $\underline{\n_0}$, $\overline{\n_0}$, $\| \n_0 - \tilde{\n} \|_{H^1}$, $\| \mathbf{u}_0 \|_{L^2}$, $\ga$, $\tilde{\n}$, and $\alpha$ such that
\be\la{2101}\ba
& \sup_{0\le t \le T} \int_0^\infty \left( \n u^2 + |\p_r \n^{\alpha-\frac{1}{2}}|^2 + |\n - \tilde{\n}|^p \right) r dr \\
& + \int_0^T \int_0^\infty \left( \n^{\alpha} \frac{u^2}{r^2} + \n^{\alpha} |\p_r u|^2 + \n^{\ga-\alpha-1} |\p_r \n^\alpha|^2 \right) r dr dt
\le C.
\ea\ee
\end{lemma}
\begin{proof}
First, multiplying $(\ref{nsqdc})_2$ by $r u$ and integrating by parts over $(0,\infty)$, we derive
\be\la{211}\ba
& \frac{1}{2} \frac{d}{dt} \int_0^\infty \n u^2 r dr
+ \alpha \int_0^\infty \n^{\alpha} \frac{u^2}{r} dr
+ \alpha \int_0^\infty \n^{\alpha} |\p_r u|^2 r dr \\
& = 2(1-\alpha) \int_0^\infty \n^\alpha u (\p_r u) dr
- \int_0^\infty \p_r(\n^\ga) u r dr.
\ea\ee
Note that $K(\n)$ satisfies
\be\la{212}\ba
\frac{d}{dt} \int_0^\infty K(\n) r dr - \int_0^\infty \p_r(\n^\ga) u r dr = 0.
\ea\ee
Combining \eqref{211} with \eqref{212} yields
\be\la{213}\ba
& \frac{d}{dt} \int_0^\infty \left( \frac{1}{2} \n u^2 + K(\n) \right) r dr
+ \alpha \int_0^\infty \n^{\alpha} \frac{u^2}{r} dr
+ \alpha \int_0^\infty \n^{\alpha} |\p_r u|^2 r dr \\
& = 2(1-\alpha) \int_0^\infty \n^\alpha u (\p_r u) dr.
\ea\ee
It follows from Young's inequality that for any $0<\ep<\alpha$,
\be\la{214}\ba
2(1-\alpha) \int_0^\infty \n^\alpha u (\p_r u) dr
\le \frac{(1-\alpha)^2}{\alpha-\ep} \int_0^\infty \n^{\alpha} \frac{u^2}{r} dr
+ (\alpha-\ep) \int_0^\infty \n^{\alpha} |\p_r u|^2 r dr.
\ea\ee
Using the fact that $\alpha>\frac{1}{2}$, we can choose sufficiently small $0<\ep<\alpha$ such that
\be\nonumber
\frac{(1-\alpha)^2}{\alpha-\ep}<\alpha-\ep.
\ee
For such $\ep$, substituting (\ref{214}) into (\ref{213}) and integrating over $(0,T)$ lead to
\be\la{215}\ba
\sup_{0\le t \le T} \int_0^\infty \left( \n u^2 + K(\n) \right) r dr
+ \int_0^T \int_0^\infty \left( \n^{\alpha} \frac{u^2}{r^2} + \n^{\alpha} |\p_r u|^2 \right) r dr dt
\le C.
\ea\ee
From \eqref{pe}, we conclude that there is a positive constant $C$ depending only on $\tilde{\n}$ and $\ga$, such that
\be\la{216}\ba
(\n - \tilde{\n})^2 \le C K(\n), \text{ if } \n < 2 \tilde{\n}, \quad
(\n - \tilde{\n})^\ga \le C K(\n), \text{ if } \n \ge 2 \tilde{\n},
\ea\ee
which together with (\ref{215}) gives
\be\la{217}\ba
\int_{( \n < 2 \tilde{\n} )} |\n - \tilde{\n}|^2 r dr
+ \int_{( \n \ge 2 \tilde{\n} )} |\n - \tilde{\n}|^\ga r dr \le C.
\ea\ee
By virtue of the condition $2\alpha-1<1<\ga$, we have
\be\la{218}\ba
\int_0^\infty |\n^{\alpha-\frac{1}{2}} - \tilde{\n}^{\alpha-\frac{1}{2}}|^2 r dr
& = \int_{( \n < 2 \tilde{\n} )} |\n^{\alpha-\frac{1}{2}} - \tilde{\n}^{\alpha-\frac{1}{2}}|^2 r dr
+ \int_{( \n \ge 2 \tilde{\n} )} |\n^{\alpha-\frac{1}{2}} - \tilde{\n}^{\alpha-\frac{1}{2}}|^2 r dr \\
& \le C \int_{( \n < 2 \tilde{\n} )} |\n - \tilde{\n}|^2 r dr
+ C \int_{( \n \ge 2 \tilde{\n} )} |\n - \tilde{\n}|^\ga r dr \\
& \le C,
\ea\ee
which shows that
\be\la{219}\ba
\sup_{0 \le t \le T} \int_0^\infty |\n^{\alpha-\frac{1}{2}} - \tilde{\n}^{\alpha-\frac{1}{2}}|^2 r dr \le C.
\ea\ee

In addition, multiplying $(\ref{yxsdfc1})$ by $w r$, integrating by parts over $(0,\infty)$, and using (\ref{212}), we obtain
\be\la{2110}\ba
\frac{1}{2} \frac{d}{d t} \int_0^\infty \n w^2 r dr
& = - \int_0^\infty \left( \p_r(\n^\ga) u r + \p_r(\n^\ga) \n^{-1} \p_r(\n^\alpha) r \right) dr \\
& = - \frac{d}{dt} \int_0^\infty K(\n) r dr
- \frac{\ga}{\alpha} \int_0^\infty \n^{\ga-\alpha-1} |\p_r \n^\alpha|^2 r dr,
\ea\ee
which yields
\be\la{2111}\ba
\sup_{0 \le t \le T} \int_0^\infty \left( \n w^2 + K(\n) \right) r dr
+ \int_0^T \int_0^\infty \n^{\ga-\alpha-1} |\p_r \n^\alpha|^2 r dr dt \le C.
\ea\ee
In view of (\ref{215}), (\ref{2111}), and Cauchy's inequality, we arrive at
\be\la{2112}\ba
\int_0^\infty |\p_r \n^{\alpha-\frac{1}{2}}|^2 r dr
\le C \int_0^\infty \n \left( u^2 + w^2 \right) r dr \le C.
\ea\ee
It follows from (\ref{218}), (\ref{2112}), and the Sobolev inequality that for any $s \in [2,\infty)$,
\be\la{2113}\ba
\| \n^{\alpha-\frac{1}{2}} - \tilde{\n}^{\alpha-\frac{1}{2}} \|_{L^s(\rr)}^2
& \le C \| \n^{\alpha-\frac{1}{2}} - \tilde{\n}^{\alpha-\frac{1}{2}} \|_{H^1(\rr)}^2 \\
& \le C \int_0^\infty |\n^{\alpha-\frac{1}{2}} - \tilde{\n}^{\alpha-\frac{1}{2}}|^2 r dr
+ C \int_0^\infty |\p_r \n^{\alpha-\frac{1}{2}}|^2 r dr \\
& \le C.
\ea\ee
For any $p \in [2,\infty)$, we use (\ref{2113}) and the fact that $\frac{1}{2}<\alpha<1$ to obtain
\be\la{2114}\ba
\int_0^\infty |\n - \tilde{\n}|^p r dr
& = \int_{( \n < 2 \tilde{\n} )} |\n - \tilde{\n}|^p r dr
+ \int_{( \n \ge 2 \tilde{\n} )} |\n - \tilde{\n}|^p r dr \\
& \le C \int_{( \n < 2 \tilde{\n} )} |\n^{\alpha-\frac{1}{2}} - \tilde{\n}^{\alpha-\frac{1}{2}}|^p r dr
+ C \int_{( \n \ge 2 \tilde{\n} )} |\n^{\alpha-\frac{1}{2}} - \tilde{\n}^{\alpha-\frac{1}{2}}|^{\frac{2p}{2\alpha-1}} r dr \\
& \le C.
\ea\ee
From (\ref{215}), (\ref{2111}), and (\ref{2114}), we conclude (\ref{2101}) holds and thus complete the proof of Lemma \ref{2l1}.
\end{proof}

\begin{lemma}\la{2l2}
For any $\xi>0$, there exists a positive constant $C$ depending only on
$T$, $\xi$, $\underline{\n_0}$, $\overline{\n_0}$, $\| \n_0 - \tilde{\n} \|_{H^1}$, $\| \mathbf{u}_0 \|_{L^2}$, $\ga$, $\tilde{\n}$, and $\alpha$ such that
\be\la{202}\ba
\sup_{0 \le t \le T} \left( \| \n r^{\xi} \|_{L^\infty(0,1)} + \| \n \|_{L^\infty(1,\infty)} \right) \le C.
\ea\ee
\end{lemma}
\begin{proof}
First, it follows from (\ref{ywsi2}), (\ref{2101}), and H\"older's inequality that for any $0 < \zeta \le 2\alpha-1$,
\be\la{221}\ba
& \| \n^{\alpha-\frac{1}{2}} r^{\zeta} \|_{L^\infty(0,1)} \\
& \le C \int_0^1 \n^{\alpha-\frac{1}{2}} r^{\zeta} dr
+ C \int_0^1 \left( |\p_r \n^{\alpha-\frac{1}{2}}| r^{\zeta} + \n^{\alpha-\frac{1}{2}} r^{\zeta-1} \right) dr \\
& \le C \int_0^1 \n^{\alpha-\frac{1}{2}} r^{\zeta-1} dr
+ C \int_0^1 |\p_r \n^{\alpha-\frac{1}{2}}| r^{\zeta} dr \\
& \le C \left( \int_0^1 \n^{\frac{4\alpha-2}{\zeta}} r dr \right)^{\frac{\zeta}{4}}
\left( \int_0^1 r^{-\frac{4-3\zeta}{4-\zeta}} dr \right)^{\frac{4-\zeta}{4}}
+ C \left( \int_0^1 |\p_r \n^{\alpha-\frac{1}{2}}|^2 r dr \right)^{\frac{1}{2}}
\left( \int_0^1 r^{2\zeta-1} dr \right)^{\frac{1}{2}} \\
& \le C + C \int_0^1 \n^{\frac{4\alpha-2}{\zeta}} r dr \\
& \le C + C \int_0^1 |\n-\tilde{\n}|^{\frac{4\alpha-2}{\zeta}} r dr + C \int_0^1 r dr  \\
& \le C,
\ea\ee
due to $2 \le \frac{4\alpha-2}{\zeta}<\infty$.

Moreover, for any $\zeta>2\alpha-1$, we deduce from (\ref{221}) that
\be\la{22101}\ba
& \| \n^{\alpha-\frac{1}{2}} r^{\zeta} \|_{L^\infty(0,1)}
\le C \| \n^{\alpha-\frac{1}{2}} r^{2\alpha-1} \|_{L^\infty(0,1)}
\le C.
\ea\ee
Hence, for any $\zeta>0$, we have
\be\la{20102}\ba
\sup_{0 \le t \le T} \| \n^{\alpha-\frac{1}{2}} r^{\zeta} \|_{L^\infty(0,1)} \le C.
\ea\ee
For any $\xi>0$, we choose $\zeta=(\alpha-\frac{1}{2}) \xi$ in (\ref{20102}) to obtain
\be\la{22103}\ba
\sup_{0 \le t \le T} \| \n r^{\xi} \|_{L^\infty(0,1)} \le C.
\ea\ee
On the other hand, we use (\ref{ywsi1}), (\ref{2101}), and H\"older's inequality to derive
\be\nonumber\ba
\| (\n - \tilde{\n})^2 \|_{L^\infty(1,\infty)}
& \le \int_1^\infty | \p_r(\n - \tilde{\n})^2 | dr \\
& \le C \int_1^\infty |\n - \tilde{\n}| \n^{\frac{3}{2}-\alpha} |\p_r \n^{\alpha-\frac{1}{2}}| dr \\
& \le C \int_1^\infty |\n - \tilde{\n}|^2 \n^{3-2\alpha} dr
+ C \int_1^\infty |\p_r \n^{\alpha-\frac{1}{2}}|^2 dr \\
& \le C \int_1^\infty \chi_{(0 < \n < 2\tilde{\n})} |\n - \tilde{\n}|^2 \n^{3-2\alpha} dr
+ C \int_1^\infty \chi_{(\n \ge 2\tilde{\n})} |\n - \tilde{\n}|^2 \n^{3-2\alpha} dr + C \\
& \le C \int_1^\infty |\n - \tilde{\n}|^2 r dr
+ C \int_1^\infty |\n - \tilde{\n}|^{5-2\alpha} r dr + C \\
& \le C,
\ea\ee
which together with Cauchy's inequality yields
\be\la{223}\ba
\| \n \|^2_{L^\infty(1,\infty)}
\le 2 \| \n - \tilde{\n} \|^2_{L^\infty(1,\infty)} + 2 \tilde{\n}^2 \le C.
\ea\ee
Combining (\ref{22103}) with (\ref{223}) gives (\ref{202}) and completes the proof of Lemma \ref{2l2}.
\end{proof}

To obtain the upper bound for $\n$, we need extra integrability of $\n u$ and $\n w$.
For $k \in (2,\infty)$, define
\be\nonumber
f(k) \triangleq 1 - \frac{2k\sqrt{k-1}-4k+4}{(k-2)^2}.
\ee
It is straightforward to verify that $f(k)$ is strictly increasing on $(2,\infty)$ and satisfies
\be\nonumber
\lim_{k \to 2^+} f(k) = \frac{1}{2}, \quad \lim_{k \to \infty} f(k) = 1.
\ee
Let $\varphi(z) : (\frac{1}{2},1) \to (2,\infty)$ denote the inverse function of $f(k)$.
Then $\varphi(z)$ is strictly increasing on $(\frac{1}{2},1)$ and satisfies
\be\nonumber
\lim_{z \to \frac{1}{2}^+} \varphi(z) = 2, \quad \lim_{z \to 1^-} \varphi(z) = \infty.
\ee

Next, we derive the extra integrability of the momentum $\n u$.

\begin{lemma}\la{2l3}
Assume that \eqref{th11} holds.
For any $k \in [3,\varphi(\alpha))$, there exists a positive constant $C$ depending only on
$T$, $k$, $\underline{\n_0}$, $\overline{\n_0}$, $\| \n_0 - \tilde{\n} \|_{H^1}$, $\| \mathbf{u}_0 \|_{H^1}$, $\ga$, $\tilde{\n}$, and $\alpha$ such that
\be\la{203}\ba
\sup_{0\le t \le T} \int_0^\infty \n |u|^{k} r dr
+ \int_0^T \int_0^\infty \left( \n^\alpha \frac{|u|^{k}}{r} + \n^\alpha |\p_r u|^2 |u|^{k-2} r \right) dr dt
\le C.
\ea\ee
\end{lemma}
\begin{proof}
For any $k \in [3,\varphi(\alpha))$, we multiply $(\ref{nsqdc})_2$ by $|u|^{k-2}u r$ and integrate by parts over $(0,\infty)$ to obtain
\be\la{231}\ba
& \frac{1}{k} \frac{d}{dt} \int_0^\infty \n |u|^{k} r dr
+ \alpha \int_0^\infty \n^\alpha \frac{|u|^k}{r} dr
+ (k-1) \alpha \int_0^\infty \n^\alpha |u|^{k-2} |\p_r u|^2 r dr \\
& = k (1-\alpha) \int_0^\infty \n^\alpha |u|^{k-2} u (\p_r u) dr
+ (k-1) \int_0^\infty \n^\ga |u|^{k-2} (\p_r u) r dr
+ \int_0^\infty \n^{\ga} |u|^{k-2} u dr.
\ea\ee
Since $3 \le k <\varphi(\alpha)$, we have
\be\la{232}
1 - \frac{2k\sqrt{k-1}-4k+4}{(k-2)^2} < \alpha,
\ee
which implies
\be\la{233}
k^2 (1-\alpha)^2 < 4(k-1)\alpha^2.
\ee
Hence, we can choose sufficiently small $\ep>0$ such that
\be\la{234}\ba
\frac{k^2(1-\alpha)^2}{4(1-2\ep)\alpha} \le (1-2\ep) (k-1) \alpha,
\ea\ee
which together with Young's inequality yields
\be\la{235}\ba
k (1-\alpha) \n^\alpha |u|^{k-2} u (\p_r u)
& \le (1-2\ep) \alpha \n^\alpha \frac{|u|^k}{r}
+ \frac{k^2(1-\alpha)^2}{4(1-2\ep)\alpha} \n^{\alpha} |u|^{k-2} |\p_r u|^2 r \\
& \le (1-2\ep) \alpha \n^\alpha \frac{|u|^k}{r}
+ (1-2\ep) (k-1) \alpha \n^{\alpha} |u|^{k-2} |\p_r u|^2 r.
\ea\ee
Substituting (\ref{235}) into (\ref{231}) and using (\ref{2101}), (\ref{202}), and Young's inequality, we derive
\be\la{236}\ba
& \frac{1}{k} \frac{d}{dt} \int_0^\infty \n |u|^{k} r dr
+ \ep \int_0^\infty \n^\alpha \frac{|u|^k}{r} dr
+ 2\ep \int_0^\infty \n^\alpha |u|^{k-2} |\p_r u|^2 r dr \\
& \le (k-1) \int_0^\infty \n^\ga |u|^{k-2} (\p_r u) r dr
+ \int_0^\infty \n^{\ga} |u|^{k-2} u dr \\
& = (k-1) \int_0^1 \n^\ga |u|^{k-2} (\p_r u) r dr + \int_0^1 \n^{\ga} |u|^{k-2} u dr \\
& \quad + (k-1) \int_1^\infty \n^\ga |u|^{k-2} (\p_r u) r dr + \int_1^\infty \n^{\ga} |u|^{k-2} u dr \\
& \le \frac{\ep}{2} \int_0^1 \n^\alpha \frac{|u|^k}{r} dr
+ \frac{\ep}{2} \int_0^1 \n^\alpha |u|^{k-2} |\p_r u|^2 r dr
+ C \int_0^1 \n^{k(\ga-\alpha)+\alpha} r^{k-1} dr \\
& \quad + (k-1) \int_1^\infty \n^\ga |u|^{k-2} (\p_r u) r dr
+ C \int_1^\infty \n^{\ga} |u|^{2} r dr + C \int_1^\infty \n^{\ga} |u|^{k} r dr \\
& \le \frac{\ep}{2} \int_0^1 \n^\alpha \frac{|u|^k}{r} dr
+ \frac{\ep}{2} \int_0^1 \n^\alpha |u|^{k-2} |\p_r u|^2 r dr
+ C \left\| \n r^{\frac{1}{k(\ga-\alpha)+\alpha}} \right\|_{L^\infty(0,1)}^{k(\ga-\alpha)+\alpha}
\int_0^1 r^{k-2} dr \\
& \quad + (k-1) \int_1^\infty \n^\ga |u|^{k-2} (\p_r u) r dr
+ C + C \int_1^\infty \n^{\ga} |u|^{k} r dr,
\ea\ee
which together with (\ref{2101}), (\ref{202}), and the fact that $k(\ga-\alpha)+\alpha>0$ gives
\be\la{2361}\ba
& \frac{1}{k} \frac{d}{dt} \int_0^\infty \n |u|^{k} r dr
+ \frac{\ep}{2} \int_0^\infty \n^\alpha \frac{|u|^k}{r} dr
+ \ep \int_0^\infty \n^\alpha |u|^{k-2} |\p_r u|^2 r dr \\
& \le (k-1) \int_1^\infty \n^\ga |u|^{k-2} (\p_r u) r dr
+ C + C \int_1^\infty \n |u|^{k} r dr.
\ea\ee
If $k \in [3,4]$, by (\ref{2101}) and (\ref{202}), we arrive at
\be\la{2362}\ba
& (k-1) \int_1^\infty \n^\ga |u|^{k-2} (\p_r u) r dr \\
& \le C \int_1^\infty \n^\ga |u|^{2k-4} r dr + C \int_1^\infty \n^\ga |\p_r u|^2 r dr \\
& \le C \int_1^\infty \n |u|^{2} r dr + C \int_1^\infty \n |u|^{k} r dr + C \int_1^\infty \n^\alpha |\p_r u|^2 r dr \\
& \le C + C \int_1^\infty \n |u|^{k} r dr + C \int_1^\infty \n^\alpha |\p_r u|^2 r dr.
\ea\ee
If $k>4$, it follows from (\ref{2101}), (\ref{202}), and Young's inequality that
\be\la{2363}\ba
& (k-1) \int_1^\infty \n^\ga |u|^{k-2} (\p_r u) r dr \\
& \le \frac{\ep}{2} \int_0^\infty \n^\alpha |u|^{k-2} |\p_r u|^2 r dr
+ C \int_1^\infty \n^{2\ga-\alpha} |u|^{k-2} r dr \\
& \le \frac{\ep}{2} \int_0^\infty \n^\alpha |u|^{k-2} |\p_r u|^2 r dr
+ C \int_1^\infty \n |u|^{2} r dr + C \int_1^\infty \n |u|^{k} r dr \\
& \le \frac{\ep}{2} \int_0^\infty \n^\alpha |u|^{k-2} |\p_r u|^2 r dr
+ C + C \int_1^\infty \n |u|^{k} r dr.
\ea\ee
Combining (\ref{2361}), (\ref{2362}), and (\ref{2363}) gives
\be\la{238}\ba
& \frac{1}{k} \frac{d}{dt} \int_0^\infty \n |u|^{k} r dr
+ \frac{\ep}{2} \int_0^\infty \n^\alpha \frac{|u|^k}{r} dr
+ \frac{\ep}{2} \int_0^\infty \n^\alpha |u|^{k-2} |\p_r u|^2 r dr \\
& \le C + C \int_0^\infty \n |u|^{k} r dr + C \int_0^\infty \n^\alpha |\p_r u|^2 r dr,
\ea\ee
which together with Gr\"onwall's inequality yields (\ref{203}) and completes the proof of Lemma \ref{2l3}.
\end{proof}

\begin{lemma}\la{2l4}
Assume that \eqref{th11} holds.
For any $k \in [3,\varphi(\alpha))$, there exists a positive constant $C$ depending only on
$T$, $k$, $\underline{\n_0}$, $\overline{\n_0}$, $\| \n_0 - \tilde{\n} \|_{H^2}$, $\| \mathbf{u}_0 \|_{H^1}$, $\ga$, $\tilde{\n}$, and $\alpha$ such that
\be\la{204}\ba
\sup_{0 \le t \le T} \int_0^\infty \n^{k(\alpha-2)+1} |\p_r\n|^{k} r dr
+ \int_0^T \int_0^\infty \n^{k(\alpha-2)+\ga+1-\alpha} |\p_r\n|^{k} r dr dt
\le C.
\ea\ee
\end{lemma}
\begin{proof}
First, by (\ref{yxsd}) and (\ref{yxsdfc1}), the effective velocity $w$ satisfies
\be\la{2401}\ba
\n w_t + \n u w_r + \frac{\ga}{\alpha} \n^{\ga+1-\alpha} (w-u) = 0.
\ea\ee
For any $k \in [3,\varphi(\alpha))$, we multiply (\ref{2401}) by $|w|^{k-2} w r$, integrate over $(0,\infty)$, and use $(\ref{nsqdc})_1$ and Young's inequality to derive
\be\la{241}\ba
& \frac{1}{k} \frac{d}{dt} \int_0^\infty \n |w|^{k} r dr
+ \frac{\ga}{\alpha} \int_0^\infty \n^{\ga+1-\alpha} |w|^{k} r dr \\
& = \frac{\ga}{\alpha} \int_0^\infty \n^{\ga+1-\alpha} |w|^{k-2} w u r dr \\
& \le \frac{\ga}{2\alpha} \int_0^\infty \n^{\ga+1-\alpha} |w|^{k} r dr
+ C \int_0^\infty \n^{\ga+1-\alpha} |u|^{k} r dr.
\ea\ee
It follows from (\ref{202}) and (\ref{203}) that
\be\la{242}\ba
& \int_0^\infty \n^{\ga+1-\alpha} |u|^{k} r dr \\
& = \int_0^1 \n^{\ga+1-\alpha} |u|^{k} r dr
+ \int_1^\infty \n^{\ga+1-\alpha} |u|^{k} r dr \\
& \le \left\| \n r^{ \frac{2}{\ga+1-2\alpha} } \right\|^{\ga+1-2\alpha}_{L^\infty(0,1)}
\int_0^1 \n^{\alpha} \frac{|u|^{k}}{r} dr
+ \| \n \|_{L^\infty(1,\infty)}^{\ga-\alpha} \int_1^\infty \n |u|^{k} r dr \\
& \le C \int_0^1 \n^{\alpha} \frac{|u|^{k}}{r} dr + C.
\ea\ee
Combining (\ref{241}) and (\ref{242}) leads to
\be\la{243}\ba
\frac{1}{k} \frac{d}{dt} \int_0^\infty \n |w|^{k} r dr
+ \frac{\ga}{2\alpha} \int_0^\infty \n^{\ga+1-\alpha} |w|^{k} r dr
\le C \int_0^1 \n^{\alpha} \frac{|u|^{k}}{r} dr + C.
\ea\ee
Integrating (\ref{243}) over $(0,T)$ and using (\ref{203}), we arrive at
\be\la{244}\ba
\sup_{0 \le t \le T} \int_0^\infty \n |w|^{k} r dr
+ \int_0^T \int_0^\infty \n^{\ga+1-\alpha} |w|^{k} r dr
\le C.
\ea\ee
By virtue of (\ref{yxsd}), (\ref{244}), (\ref{203}), and (\ref{242}), we have
\be\la{245}\ba
\sup_{0 \le t \le T} \int_0^\infty \n^{1-k} |\p_r \n^\alpha|^{k} r dr
+ \int_0^T \int_0^\infty \n^{\ga+1-\alpha-k} |\p_r \n^\alpha|^{k} r dr
\le C,
\ea\ee
which implies (\ref{204}) and finishes the proof of Lemma \ref{2l4}.
\end{proof}

\begin{lemma}\la{2l5}
Assume that \eqref{th11} holds.
Then there exists a positive constant $C$ depending only on
$T$, $\underline{\n_0}$, $\overline{\n_0}$, $\| \n_0 - \tilde{\n} \|_{H^2}$, $\| \mathbf{u}_0 \|_{H^1}$, $\ga$, $\tilde{\n}$, and $\alpha$ such that
\be\la{205}\ba
\sup_{0 \le t \le T} \| \n \|_{L^\infty(0,\infty)} \le C.
\ea\ee
\end{lemma}
\begin{proof}
For $k \in [3,\varphi(\alpha))$, it follows from (\ref{202}), (\ref{204}), and Young's inequality that
\be\la{251}\ba
\| \n \|_{L^\infty(0,1)}
& \le \int_0^1 \n dr + \int_0^1 |\p_r \n| dr \\
& \le \| \n r^{\frac{1}{2}} \|_{L^\infty(0,1)} \int_0^1 r^{-\frac{1}{2}} dr
+ C \int_0^1 \n^{k(\alpha-2)+1} |\p_r\n|^{k} r dr \\
& \quad + C \int_0^1 \n^{ \frac{(2-\alpha)k-1}{k-1} } r^{-\frac{1}{k-1}} dr \\
& \le C + C \left\| \n r^{ \frac{1}{ 2((2-\alpha)k-1) } } \right\|_{L^\infty(0,1)}^{ \frac{(2-\alpha)k-1}{k-1} } \int_0^1 r^{-\frac{3}{2(k-1)}} dr \\
& \le C,
\ea\ee
which together with (\ref{202}) yields (\ref{205}).
\end{proof}

Next, we derive the crucial lower bound for the density.
Note that, for $k_1$ defined in (\ref{k1}), we have
\be\nonumber
1 - \frac{2k_1\sqrt{k_1-1}-4k_1+4}{(k_1-2)^2} = 1 - \frac{2}{k_1}.
\ee
Moreover, for any $k>k_1$,
\be\nonumber
1 - \frac{2k\sqrt{k-1}-4k+4}{(k-2)^2} < 1 - \frac{2}{k}.
\ee
Hence, for any $1-\frac{2}{k_1}<\alpha<1$, there exists $k>k_1$ such that
\be\la{ewk}
1 - \frac{2k\sqrt{k-1}-4k+4}{(k-2)^2} < \alpha < 1 - \frac{2}{k}.
\ee

\begin{lemma}\la{2l6}
Assume that \eqref{th11} holds.
Then there exists a positive constant $C$ depending only on
$T$, $\underline{\n_0}$, $\overline{\n_0}$, $\| \n_0 - \tilde{\n} \|_{H^2}$, $\| \mathbf{u}_0 \|_{H^1}$, $\ga$, $\tilde{\n}$, and $\alpha$ such that
\be\la{206}\ba
\sup_{0 \le t \le T} \| \n^{-1} \|_{L^\infty(0,\infty)} \le C.
\ea\ee
\end{lemma}
\begin{proof}
First, from $(\ref{nsqdc})_1$, we deduce that $\n^{-1}$ satisfies
\be\la{261}\ba
(\n^{-1})_t + \p_r (\n^{-1}) u - \n^{-1} \p_r u - \frac{1}{r} \n^{-1} u = 0.
\ea\ee
For any $m \ge 1$, multiplying (\ref{261}) by $\frac{1}{m} \n^{-\frac{1}{m} + 1}$ yields
\be\la{262}\ba
(\n^{-\frac{1}{m}})_t + \p_r (\n^{-\frac{1}{m}}) u - \frac{1}{m} \n^{-\frac{1}{m}} \p_r u - \frac{1}{mr} \n^{-\frac{1}{m}} u = 0,
\ea\ee
which implies
\be\la{263}\ba
(\n^{-\frac{1}{m}} - \tilde{\n}^{-\frac{1}{m}})_t + \p_r (\n^{-\frac{1}{m}} - \tilde{\n}^{-\frac{1}{m}}) u - \frac{1}{m} \n^{-\frac{1}{m}} \p_r u - \frac{1}{mr} \n^{-\frac{1}{m}} u = 0.
\ea\ee
Multiplying (\ref{263}) by $2 (\n^{-\frac{1}{m}} - \tilde{\n}^{-\frac{1}{m}}) r$ and integrating by parts over $(0,\infty)$, we derive
\be\la{264}\ba
& \frac{d}{dt} \int_0^\infty (\n^{-\frac{1}{m}} - \tilde{\n}^{-\frac{1}{m}})^2 r dr \\
& = - \int_0^\infty \p_r (\n^{-\frac{1}{m}} - \tilde{\n}^{-\frac{1}{m}})^2 u r dr
+ \frac{2}{m} \int_0^\infty \n^{-\frac{1}{m}} (\p_r u) (\n^{-\frac{1}{m}} - \tilde{\n}^{-\frac{1}{m}}) r dr \\
& \quad + \frac{2}{m} \int_0^\infty \n^{-\frac{1}{m}} u (\n^{-\frac{1}{m}} - \tilde{\n}^{-\frac{1}{m}}) dr \\
& = - \left( 2 + \frac{2}{m} \right) \int_0^\infty \p_r(\n^{-\frac{1}{m}}) (\n^{-\frac{1}{m}} - \tilde{\n}^{-\frac{1}{m}}) u r dr
- \frac{2}{m} \int_0^\infty \n^{-\frac{1}{m}} \p_r(\n^{-\frac{1}{m}} - \tilde{\n}^{-\frac{1}{m}}) u r dr \\
& \le C \int_0^\infty \n^{-\frac{1}{m}-1} |\p_r \n| |\n^{-\frac{1}{m}} - \tilde{\n}^{-\frac{1}{m}}| |u| r dr
+ C \int_0^\infty \n^{-\frac{2}{m}-1} |\p_r \n| |u| r dr \\
& \le C \int_{(0<\n<\frac{\tilde{\n}}{2})} \n^{-\frac{2}{m}-1} |\p_r \n| |u| r dr
+ C \int_{(\n \ge \frac{\tilde{\n}}{2})} |\p_r \n| |u| r dr.
\ea\ee
On the one hand, we use (\ref{2101}), (\ref{205}), and H\"older's inequality to obtain
\be\la{265}\ba
\int_{(\n \ge \frac{\tilde{\n}}{2})} |\p_r \n| |u| r dr
& \le C \int_{(\n \ge \frac{\tilde{\n}}{2})} |\p_r \n^{\alpha-\frac{1}{2}}|^2 r dr
+ C \int_{(\n \ge \frac{\tilde{\n}}{2})} \n |u|^2 r dr
\le C.
\ea\ee
On the other hand, for $k$ determined in (\ref{ewk}), from (\ref{203}), (\ref{204}), and Young's inequality, we conclude that
\be\la{266}\ba
& \int_{(0<\n<\frac{\tilde{\n}}{2})} \n^{-\frac{2}{m}-1} |\p_r \n| |u| r dr \\
& \le C \int_{(0<\n<\frac{\tilde{\n}}{2})} \n^{k(\alpha-2)+1} |\p_r\n|^{k} r dr
+ C \int_{(0<\n<\frac{\tilde{\n}}{2})} \n |u|^{k} r dr
+ C \int_{(0<\n<\frac{\tilde{\n}}{2})} \n^{-\frac{m(k\alpha-k+2)+2k}{m(k-2)}} r dr \\
& \le C + C \int_{(0<\n<\frac{\tilde{\n}}{2})} \n^{-\frac{m(k\alpha-k+2)+2k}{m(k-2)}} r dr.
\ea\ee
Since $\alpha<1-\frac{2}{k}$, we choose $m$ such that
\be\la{2660}
m \ge \max \left\{ 1, \frac{2k}{k-2-k\alpha} \right\},
\ee
which implies
\be\la{267}\ba
\frac{m(k\alpha-k+2)+2k}{m(k-2)} \le 0.
\ea\ee
Combining (\ref{267}), (\ref{205}), and (\ref{2101}) leads to
\be\la{2670}\ba
\int_{(0<\n<\frac{\tilde{\n}}{2})} \n^{-\frac{m(k\alpha-k+2)+2k}{m(k-2)}} r dr
& \le C \int_0^\infty \chi_{(0<\n<\frac{\tilde{\n}}{2})} r dr \\
& = C \int_0^\infty \chi_{(\frac{\tilde{\n}}{2}<\tilde{\n}-\n<\tilde{\n})} r dr \\
& \le C \int_0^\infty |\n - \tilde{\n}|^2 r dr \le C,
\ea\ee
Substituting (\ref{265}), (\ref{266}), and (\ref{2670}) into (\ref{264}) yields
\be\la{268}\ba
\frac{d}{dt} \int_0^\infty (\n^{-\frac{1}{m}} - \tilde{\n}^{-\frac{1}{m}})^2 r dr
& \le C.
\ea\ee
Moreover, the condition $\n_0 \ge \underline{\n_0}>0$ along with Taylor's formula shows
\be\la{269}\ba
\n_0^{\frac{1}{m}} = \tilde{\n}^{\frac{1}{m}} + \frac{1}{m} \tilde{\n}_0^{\frac{1}{m}-1} (\n_0 - \tilde{\n}),
\ea\ee
where $\tilde{\n}_0$ lies between $\n_0$ and $\tilde{\n}$.
This implies that
\be\la{2610}\ba
|\n_0^{-\frac{1}{m}} - \tilde{\n}^{-\frac{1}{m}}| = (\n_0 \tilde{\n})^{-\frac{1}{m}} |\n_0^{\frac{1}{m}} - \tilde{\n}^{\frac{1}{m}}| \le C |\n_0 - \tilde{\n}|.
\ea\ee
Hence, (\ref{2610}) and (\ref{th12}) ensure that
\be\la{2611}\ba
\int_0^\infty (\n_0^{-\frac{1}{m}} - \tilde{\n}^{-\frac{1}{m}})^2 r dr
\le C \int_0^\infty |\n_0 - \tilde{\n}|^2 r dr
\le C.
\ea\ee
Integrating (\ref{268}) over $(0,T)$ and using (\ref{2611}), we arrive at
\be\la{2612}\ba
\sup_{0 \le t \le T} \int_0^\infty (\n^{-\frac{1}{m}} - \tilde{\n}^{-\frac{1}{m}})^2 r dr
\le C.
\ea\ee
In view of (\ref{ywsi1}), (\ref{2101}), (\ref{204}), (\ref{2670}), and (\ref{2612}), we derive
\be\la{2613}\ba
& \| (\n^{-\frac{1}{m}} - \tilde{\n}^{-\frac{1}{m}})^2 \|_{L^\infty(1,\infty)} \\
& \le \int_1^\infty |\p_r (\n^{-\frac{1}{m}} - \tilde{\n}^{-\frac{1}{m}})^2| dr \\
& \le C \int_1^\infty |\n^{-\frac{1}{m}} - \tilde{\n}^{-\frac{1}{m}}| \n^{-\frac{1}{m}-1} |\p_r \n| dr \\
& \le C \int_1^\infty \chi_{(0 < \n < \frac{\tilde{\n}}{2})} |\n^{-\frac{1}{m}} - \tilde{\n}^{-\frac{1}{m}}| \n^{-\frac{1}{m}-1} |\p_r \n| dr
+ C \int_1^\infty \chi_{(\n \ge \frac{\tilde{\n}}{2})} |\n^{-\frac{1}{m}} - \tilde{\n}^{-\frac{1}{m}}| \n^{-\frac{1}{m}-1} |\p_r \n| dr \\
& \le C \int_1^\infty \chi_{(0 < \n < \frac{\tilde{\n}}{2})} \n^{-\frac{2}{m}-1} |\p_r \n| dr
+ C \int_1^\infty \chi_{(\n \ge \frac{\tilde{\n}}{2})} |\n^{-\frac{1}{m}} - \tilde{\n}^{-\frac{1}{m}}| |\p_r \n^{\alpha-\frac{1}{2}}| dr \\
& \le C \left( \int_0^\infty \n^{k(\alpha-2)+1} |\p_r\n|^{k} r dr \right)^{\frac{1}{k}}
\left( \int_1^\infty \chi_{(0 < \n < \frac{\tilde{\n}}{2})} \n^{-\frac{m(k\alpha-k+1)+2k}{m(k-1)}} r dr \right)^{\frac{k-1}{k}} \\
& \quad + C \int_0^\infty |\n^{-\frac{1}{m}} - \tilde{\n}^{-\frac{1}{m}}|^2 r dr
+ C \int_0^\infty |\p_r \n^{\alpha-\frac{1}{2}}|^2 r dr \\
& \le C,
\ea\ee
where we have used the following fact
\be\la{2614}\ba
\frac{m(k\alpha-k+1)+2k}{m(k-1)} \le 0,
\ea\ee
due to (\ref{2660}).

From (\ref{2613}), we conclude that
\be\la{2615}\ba
\| \n^{-1} \|_{L^\infty(1,\infty)}
& = \| \n^{-\frac{1}{m}} \|_{L^\infty(1,\infty)}^m
\le \left( \| \n^{-\frac{1}{m}} - \tilde{\n}^{-\frac{1}{m}} \|_{L^\infty(1,\infty)} + \tilde{\n}^{-\frac{1}{m}} \right)^m
\le C.
\ea\ee
Next, we estimate $\| \n^{-1} \|_{L^\infty(0,1)}$ in the Lagrangian coordinates.

By (\ref{2101}), there exists a positive constant $C_0$ depending only on $\n_0$, $\mathbf{u}_0$, $\ga$, $\tilde{\n}$, $\alpha$, and $T$ such that
\be\la{261501}\ba
\sup_{0 \le t \le T} \int_0^\infty |\n - \tilde{\n}|^2 r dr \le C_0.
\ea\ee
This, combined with H\"older's inequality, implies that for any $0 < r < \infty$,
\be\la{261502}\ba
\int_0^r \n s ds
& = \int_0^r (\n - \tilde{\n}) s ds + \tilde{\n} \int_0^r s ds \\
& \ge - \left( \int_0^r |\n - \tilde{\n}|^2 s ds \right)^{\frac{1}{2}} \left( \int_0^r s ds \right)^{\frac{1}{2}}
+ \frac{1}{2} \tilde{\n} r^2 \\
& \ge - C_0^{\frac{1}{2}} r + \frac{1}{2} \tilde{\n} r^2.
\ea\ee
We choose
\be\la{261503}
r_0 \triangleq 1 + \frac{2}{\tilde{\n}} (1 + C_0^{\frac{1}{2}}),
\ee
and set
\be\la{2616}\ba
y_0 \triangleq \int_0^{r_0} \n s ds.
\ea\ee
In view of (\ref{205}), (\ref{261502}), and (\ref{261503}), we have
\be\la{261504}
1 \le y_0 \le C.
\ee
For any $r\in(0,1)$, with $y$ defined by (\ref{lzb1}), we use (\ref{204}), (\ref{lzb1}), (\ref{lzb2}), (\ref{2615}), and (\ref{ywsi2}) to obtain
\be\la{2618}\ba
\n^{-1}(r,t) = \n^{-1}(y,t)
& \le \frac{1}{y_0} \int_0^{y_0} \n^{-1} dy
+ \int_0^{y_0} |\p_y(\n^{-1})| dy \\
& \le \int_0^{r_0} \n^{-1} \n r dr + C \int_0^{r_0} \n^{-2} |\p_r \n| dr \\
& \le C + C \| \n^{-1} \|^{\alpha+\frac{1}{k}}_{L^\infty(0,r_0)}
\left( \int_0^\infty \n^{k(\alpha-2)+1} |\p_r\n|^{k} r dr \right)^{\frac{1}{k}}
\left( \int_0^{r_0} r^{- \frac{1}{k-1}} dr \right)^{\frac{k-1}{k}} \\
& \le C + C \| \n^{-1} \|^{\alpha+\frac{1}{k}}_{L^\infty(0,r_0)} \\
& \le C + C \| \n^{-1} \|^{\alpha+\frac{1}{k}}_{L^\infty(0,1)},
\ea\ee
owing to $k>2$.

Thus, by Young's inequality and the fact that $\alpha<1-\frac{1}{k}$, we arrive at
\be\la{2619}\ba
\sup_{0 \le t \le T} \| \n^{-1}(r,t) \|_{L^\infty(0,1)} \le C,
\ea\ee
which together with (\ref{2615}) yields (\ref{206}) and completes the proof of Lemma \ref{2l6}.
\end{proof}

Next, under assumptions \eqref{th121}, \eqref{th12}, and \eqref{th122}, we derive the upper and lower bounds for the density.

Note that when $\alpha$ satisfies (\ref{th121}), there exists $k \in [3,\infty)$ such that
\be\la{2k2}\ba
1 - \frac{2k\sqrt{k-1}-4k+4}{(k-2)^2} < \alpha < 1 - \frac{1}{k}.
\ea\ee

\begin{lemma}\la{2l7}
Assume that \eqref{th121} holds and the initial data satisfy \eqref{th12} and \eqref{th122}.
Then there exists a positive constant $C$ depending only on
$\eta$, $T$, $\underline{\n_0}$, $\overline{\n_0}$, $\| \n_0 - \tilde{\n} \|_{H^2}$, $\| \mathbf{u}_0 \|_{H^1}$,
$\| |x|^{\frac{\eta}{2}} (\n_0 - \tilde{\n}) \|_{L^2}$, $\| |x|^{\frac{\eta}{2}} \na \n_0 \|_{L^2}$, $\| |x|^{\frac{\eta}{2}} \mathbf{u}_0 \|_{L^2}$,
$\ga$, $\tilde{\n}$, and $\alpha$ such that
\be\la{207}\ba
\sup_{0 \le t \le T} \left( \| \n \|_{L^\infty(0,\infty)} + \| \n^{-1} \|_{L^\infty(0,\infty)} \right) \le C.
\ea\ee
\end{lemma}
\begin{proof}
First, following arguments similar to those in Lemmas \ref{2l5} and \ref{2l6}, we obtain
\be\la{271}\ba
\sup_{0 \le t \le T} \left( \| \n \|_{L^\infty(0,\infty)} + \| \n^{-1} \|_{L^\infty(0,1)} \right) \le C.
\ea\ee
Moreover, choosing $k=3$ in (\ref{203}) and (\ref{244}) implies
\be\la{272}\ba
\sup_{0\le t \le T} \int_0^\infty \left( \n |u|^{3} + \n |w|^3 \right) r dr
+ \int_0^T \int_0^\infty \left( \n^\alpha \frac{|u|^{3}}{r} + \n^\alpha |\p_r u|^2 |u| r \right) dr dt
\le C.
\ea\ee
Multiplying $(\ref{nsqdc})_2$ by $r^{1+\eta} u$ and integrating by parts over $(0,\infty)$, we arrive at
\be\la{273}\ba
& \frac{1}{2} \frac{d}{dt} \int_0^\infty \n |u|^2 r^{1+\eta} dr
+ \alpha \int_0^\infty \n^\alpha |\p_r u|^2 r^{1+\eta} dr
+ ( (1+\eta)\alpha - \eta ) \int_0^\infty \n^\alpha |u|^2 r^{\eta-1} dr \\
& = \frac{\eta}{2} \int_0^\infty \n u^3 r^{\eta} dr
- ( (2+\eta) \alpha -2) \int_0^\infty \n^\alpha u (\p_r u) r^{\eta} dr
- \int_0^\infty \p_r(\n^\ga - \tilde{\n}^\ga) u r^{1+\eta} dr.
\ea\ee
From $(\ref{nsqdc})_1$ and the definition of $K(\n)$, we conclude that
\be\la{274}\ba
(K(\n))_t + r^{-1} (K(\n) u r)_r + r^{-1} (\n^\ga - \tilde{\n}^\ga) (u r)_r = 0.
\ea\ee
Multiplying (\ref{274}) by $r^{1+\eta}$ and integrating by parts over $(0,\infty)$ lead to
\be\la{275}\ba
\frac{d}{dt} \int_0^\infty K(\n) r^{1+\eta} dr
= \eta \int_0^\infty K(\n) u r^\eta dr
+ \int_0^\infty \p_r(\n^\ga - \tilde{\n}^\ga) u r^{1+\eta} dr
+ \eta \int_0^\infty (\n^\ga - \tilde{\n}^\ga) u r^\eta dr.
\ea\ee
Substituting (\ref{275}) into (\ref{273}) yields
\be\la{276}\ba
& \frac{d}{dt} \int_0^\infty \left( \frac{1}{2} \n |u|^2 + K(\n) \right) r^{1+\eta} dr
+ \alpha \int_0^\infty \n^\alpha |\p_r u|^2 r^{1+\eta} dr
+ ( (1+\eta) \alpha - \eta ) \int_0^\infty \n^\alpha |u|^2 r^{\eta-1} dr \\
& = \frac{\eta}{2} \int_0^\infty \n u^3 r^\eta dr
- ( (2+\eta) \alpha -2) \int_0^\infty \n^\alpha u (\p_r u) r^\eta dr
+ \eta \int_0^\infty ( K(\n) + \n^\ga - \tilde{\n}^\ga) u r^\eta dr.
\ea\ee
It follows from (\ref{271}) and (\ref{272}) that
\be\la{277}\ba
\int_0^\infty \n u^3 r^\eta dr
& = \int_0^1 \n u^3 r^\eta dr + \int_1^\infty \n u^3 r^\eta dr \\
& \le C \int_0^1 \left( \n^\alpha \frac{|u|^3}{r} + \n |u|^3 r \right) dr
+ C \int_1^\infty \n |u|^3 r dr \\
& \le C + C \int_0^\infty \n^\alpha \frac{|u|^3}{r} dr.
\ea\ee
Using (\ref{271}), (\ref{2670}), (\ref{2101}), and H\"older's inequality, we derive
\be\la{278}\ba
& \int_0^\infty \n^\alpha u (\p_r u) r^\eta dr \\
& = \int_0^1 \n^\alpha u (\p_r u) r^\eta dr + \int_1^\infty \n^\alpha u (\p_r u) r^\eta dr \\
& \le C \int_0^1 \left( \n^\alpha \frac{|u|^2}{r} + \n^\alpha |\p_r u|^2 r \right) dr
+ \int_1^\infty \chi_{(\n \ge \frac{\tilde{\n}}{2})} \n^\alpha |u| |\p_r u| r dr
+ \int_1^\infty \chi_{(\n<\frac{\tilde{\n}}{2})} \n^\alpha |u| |\p_r u| r dr \\
& \le C \int_0^1 \left( \n^\alpha \frac{|u|^2}{r} + \n^\alpha |\p_r u|^2 r \right) dr
+ C \int_0^\infty \left( \n |u|^2 + \n^\alpha  |\p_r u|^2 \right) r dr \\
& \quad + C \left( \int_0^\infty \n^\alpha |\p_r u|^2 |u| r dr \right)^{\frac{1}{2}}
\left( \int_0^\infty \n^{2 \alpha} |u|^2 r dr \right)^{\frac{1}{4}}
\left( \int_0^\infty \chi_{(\n<\frac{\tilde{\n}}{2})} r dr \right)^{\frac{1}{4}} \\
& \le C + C \int_0^\infty \left( \n^\alpha \frac{|u|^2}{r} + \n^\alpha |\p_r u|^2 r \right) dr
+ C \int_0^\infty \n^\alpha |\p_r u|^2 |u| r dr.
\ea\ee
Recalling (\ref{pe}), we have
\be\la{279}\ba
K(\n) = \frac{1}{\ga-1} \left( \n^\ga - \ga \n (\tilde{\n})^{\ga-1} \right) + \tilde{\n}^\ga.
\ea\ee
This, combined with (\ref{271}), (\ref{2101}), (\ref{2670}), and H\"older's inequality, yields
\be\la{2710}\ba
& \int_0^\infty ( K(\n) + \n^\ga - \tilde{\n}^\ga) u r^\eta dr \\
& = \int_0^1 ( K(\n) + \n^\ga - \tilde{\n}^\ga) u r^\eta dr
+ \int_1^\infty ( K(\n) + \n^\ga - \tilde{\n}^\ga) u r^\eta dr \\
& \le C \int_0^1 \n |u| dr
+ C \int_1^\infty \chi_{(\n<\frac{\tilde{\n}}{2})} ( K(\n) + \n^\ga - \tilde{\n}^\ga) u r^\eta dr \\
& \quad + C \int_1^\infty \chi_{(\n \ge \frac{\tilde{\n}}{2})} ( K(\n) + \n^\ga - \tilde{\n}^\ga) u r^\eta dr \\
& \le C + C \int_0^1 \n^\alpha \frac{|u|^2}{r} dr
+ C \int_1^\infty \chi_{(\n<\frac{\tilde{\n}}{2})} \n |u| r dr
+ C \int_1^\infty \chi_{(\n \ge \frac{\tilde{\n}}{2})} |\n - \tilde{\n}| |u| r dr \\
& \le C + C \int_0^1 \n^\alpha \frac{|u|^2}{r} dr
+ C \left( \int_0^\infty \n |u|^2 r dr \right)^{\frac{1}{2}}
\left( \int_0^\infty \chi_{(\n<\frac{\tilde{\n}}{2})} r dr \right)^{\frac{1}{2}} \\
& \quad + C \left( \int_0^\infty |\n - \tilde{\n}|^2 r dr \right)^{\frac{1}{2}}
\left( \int_0^\infty \n |u|^2 r dr \right)^{\frac{1}{2}} \\
& \le C + C \int_0^\infty \n^\alpha \frac{|u|^2}{r} dr.
\ea\ee
Putting (\ref{277}), (\ref{278}), and (\ref{2710}) into (\ref{276}) leads to
\be\la{2711}\ba
& \frac{d}{dt} \int_0^\infty \left( \frac{1}{2} \n |u|^2 + K(\n) \right) r^{1+\eta} dr
+ \alpha \int_0^\infty \n^\alpha |u_r|^2 r^{1+\eta} dr
+ ( (1+\eta) \alpha - \eta ) \int_0^\infty \n^\alpha |u|^2 r^{\eta-1} dr \\
& \le C + C \int_0^\infty \left( \n^\alpha \frac{|u|^2}{r} + \n^\alpha  |\p_r u|^2 r \right) dr
+ C \int_0^\infty \left( \n^\alpha \frac{|u|^3}{r} + \n^\alpha |\p_r u|^2 |u| r \right) dr.
\ea\ee
Integrating (\ref{2711}) over $(0,T)$ and using (\ref{2101}) and (\ref{272}), we arrive at
\be\la{2712}\ba
& \sup_{0 \le t \le T} \int_0^\infty \left( \n |u|^2 + K(\n) \right) r^{1+\eta} dr \le C.
\ea\ee
In addition, multiplying (\ref{yxsdfc1}) by $w r^{1+\eta}$ and integrating by parts over $(0,\infty)$, we obtain
\be\nonumber\ba
\frac{1}{2} \frac{d}{dt} \int_0^\infty \n w^2 r^{1+\eta} dr
+ \int_0^\infty \p_r(\n^\ga) u r^{1+\eta} dr + \int_0^\infty \p_r(\n^\ga) \n^{-1} \p_r(\n^\alpha) r^{1+\eta} dr
= \frac{\eta}{2} \int_0^\infty \n u w^2 r^\eta dr,
\ea\ee
which together with (\ref{271}), (\ref{272}), (\ref{275}), and (\ref{279}) gives
\be\la{2713}\ba
& \frac{d}{dt} \int_0^\infty \left( \frac{1}{2} \n |w|^2 + K(\n) \right) r^{1+\eta} dr
+ \frac{\ga}{\alpha} \int_0^\infty \n^{\ga-\alpha-1} |\p_r \n^\alpha|^2 r^{1+\eta} dr \\
& = \frac{\eta}{2} \int_0^\infty \n u w^2 r^\eta dr
+ \eta \int_0^\infty \left( K(\n) + \n^\ga - \tilde{\n}^\ga \right) u r^\eta dr \\
& \le C \int_0^1 \left( \n \frac{|u|^3}{r} + \n |w|^3 r^{\frac{1+3\eta}{2}} \right) dr
+ C \int_1^\infty \n (|u|^3 + |w|^3) r dr + C
+ C \int_0^\infty \n^\alpha \frac{|u|^2}{r} dr \\
& \le C + C \int_0^\infty \n^\alpha \frac{|u|^3}{r} dr
+ C \int_0^\infty \n^\alpha \frac{|u|^2}{r} dr,
\ea\ee
where in the last inequality we have used the fact that $\frac{1}{3} \le \eta \le 1$ and (\ref{272}).

Integrating (\ref{2713}) over $(0,T)$ and using (\ref{2101}) and (\ref{272}) yield
\be\la{2714}\ba
\sup_{0 \le t \le T} \int_0^\infty \left( \n |w|^2 + K(\n) \right) r^{1+\eta} dr \le C.
\ea\ee
Combining this with (\ref{2712}) and Young's inequality implies
\be\la{2715}\ba
\int_0^\infty |\p_r \n^{\alpha-\frac{1}{2}}|^2 r^{1+\eta} dr
\le C \int_0^\infty \n ( u^2 + w^2) r^{1+\eta} dr \le C.
\ea\ee
For $k$ as in (\ref{2k2}), we use (\ref{ywsi1}), (\ref{204}), (\ref{271}), (\ref{2715}), and Young's inequality to obtain
\be\nonumber\ba
& \| (\n^{-1} - \tilde{\n}^{-1})^2 \|_{L^\infty(1,\infty)} \\
& \le \int_1^\infty |\p_r (\n^{-1} - \tilde{\n}^{-1})^2| dr \\
& \le C \int_1^\infty \n^{-2} |\n^{-1} - \tilde{\n}^{-1}| |\p_r \n| dr \\
& \le C \int_1^\infty \chi_{(0 < \n < \frac{\tilde{\n}}{2})} \n^{-2} |\n^{-1} - \tilde{\n}^{-1}| |\p_r \n| dr
+ C \int_1^\infty \chi_{(\n \ge \frac{\tilde{\n}}{2})} \n^{-2} |\n^{-1} - \tilde{\n}^{-1}| |\p_r \n| dr \\
& \le C \int_1^\infty \chi_{(0 < \n < \frac{\tilde{\n}}{2})} (\n^{-1} - \tilde{\n}^{-1})^3 |\p_r \n| dr
+ C \int_1^\infty \chi_{(\n \ge \frac{\tilde{\n}}{2})} |\p_r \n^{\alpha-\frac{1}{2}}| dr \\
& \le C \| \n^{-1} - \tilde{\n}^{-1} \|_{L^\infty(1,\infty)}^{1+\alpha+\frac{1}{k}}
\int_1^\infty \chi_{(0 < \n < \frac{\tilde{\n}}{2})} (\n^{-1} - \tilde{\n}^{-1})^{2-\alpha-\frac{1}{k}} |\p_r \n| dr \\
& \quad + C \int_1^\infty \left( |\p_r \n^{\alpha-\frac{1}{2}}|^2 r^{1+\eta} + r^{-1-\eta} \right) dr \\
& \le C \| \n^{-1} - \tilde{\n}^{-1} \|_{L^\infty(1,\infty)}^{1+\alpha+\frac{1}{k}}
\int_1^\infty \chi_{(0 < \n < \frac{\tilde{\n}}{2})} \n^{\alpha-2+\frac{1}{k}} |\p_r \n| dr
+ C \\
& \le C \| \n^{-1} - \tilde{\n}^{-1} \|_{L^\infty(1,\infty)}^{1+\alpha+\frac{1}{k}}
\left( \int_1^\infty \n^{k(\alpha-2)+1} |\p_r\n|^{k} dr \right)^{\frac{1}{k}}
\left( \int_1^\infty \chi_{(0 < \n < \frac{\tilde{\n}}{2})} dr \right)^{\frac{k-1}{k}}
+ C \\
& \le C \| \n^{-1} - \tilde{\n}^{-1} \|_{L^\infty(1,\infty)}^{1+\alpha+\frac{1}{k}} + C \\
& \le \frac{1}{2} \| \n^{-1} - \tilde{\n}^{-1} \|_{L^\infty(1,\infty)}^{2} + C,
\ea\ee
which yields
\be\la{2716}\ba
\| \n^{-1} - \tilde{\n}^{-1} \|_{L^\infty(1,\infty)}^2 \le C.
\ea\ee
This implies that
\be\la{2717}\ba
\| \n^{-1} \|_{L^\infty(1,\infty)} \le \| \n^{-1} - \tilde{\n}^{-1} \|_{L^\infty(1,\infty)} + \tilde{\n}^{-1} \le C.
\ea\ee
Combining (\ref{271}) with (\ref{2717}) gives (\ref{207}) and completes the proof of Lemma \ref{2l7}.
\end{proof}

\section{Upper Bound and Lower Bound of the Density ($N=3$)}
This section is devoted to deriving both the upper bound and the positive lower bound of the density for the case $\OM=\mathbb{R}^3$.
Suppose that the initial data $(\n_0,\mathbf{u}_0)$ satisfies (\ref{th12}).
Let $(\n,\mathbf{u})$ be a spherically symmetric classical solution to (\ref{ns})--(\ref{qkjbjtj}), (\ref{nxxs}) on $\mathbb{R}^3 \times (0,T]$, provided by Lemmas \ref{lct} and \ref{dcx}.

We first establish the standard energy estimate and the BD entropy estimate.
\begin{lemma}\la{3l1}
There exists a positive constant $C$ depending only on
$T$, $\underline{\n_0}$, $\overline{\n_0}$, $\| \n_0 - \tilde{\n} \|_{H^1}$, $\| \mathbf{u}_0 \|_{L^2}$, $\ga$, $\tilde{\n}$, and $\alpha$ such that
\be\la{301}\ba
& \sup_{0\le t \le T} \int_0^\infty \left( \n u^2 + |\p_r \n^{\alpha-\frac{1}{2}}|^2 + |\sqrt{\n} - \sqrt{\tilde{\n}}|^2 + |\n^{\alpha-\frac{1}{2}} - \tilde{\n}^{\alpha-\frac{1}{2}}|^6 \right) r^{2} dr \\
& + \int_0^T \int_0^\infty \left( \n^{\alpha} \frac{u^2}{r^2} + \n^{\alpha} |\p_r u|^2 + \n^{\ga-\alpha-1} |\p_r \n^\alpha|^2 \right) r^{2} dr dt
\le C.
\ea\ee
\end{lemma}
\begin{proof}
First, arguing as in the proof of Lemma \ref{2l1}, we can obtain the standard energy estimates
\be\la{311}\ba
& \frac{d}{dt} \int_0^\infty \left( \frac{1}{2} \n u^2 + K(\n) \right) r^2 dr
+ 2(2\alpha-1) \int_0^\infty \n^{\alpha} u^2 dr
+ \alpha \int_0^\infty \n^{\alpha} |\p_r u|^2 r^2 dr \\
& = 4(1-\alpha) \int_0^\infty \n^\alpha u (\p_r u) r dr.
\ea\ee
By virtue of Young's inequality, we have for any $0<\ep<\alpha$,
\be\la{312}\ba
4(1-\alpha) \int_0^\infty \n^\alpha u (\p_r u) r dr
\le \frac{4(1-\alpha)^2}{\alpha-\ep} \int_0^\infty \n^{\alpha} u^2 dr
+ (\alpha-\ep) \int_0^\infty \n^{\alpha} |\p_r u|^2 r^2 dr.
\ea\ee
Since $\alpha>\frac{2}{3}$, we can choose sufficiently small $0<\ep<\alpha$ such that
\be\nonumber
\frac{4(1-\alpha)^2}{\alpha-\ep}<2(2\alpha-1) - \ep.
\ee
For such $\ep$, we conclude from (\ref{311}) and (\ref{312}) that
\be\la{313}\ba
\sup_{0\le t \le T} \int_0^\infty \left( \n u^2 + K(\n) \right) r^2 dr
+ \int_0^T \int_0^\infty \left( \n^{\alpha} \frac{u^2}{r^2} + \n^{\alpha} |\p_r u|^2 \right) r^2 dr dt
\le C.
\ea\ee
Combining (\ref{216}) with (\ref{313}) leads to
\be\la{314}\ba
\int_{( \n < 2 \tilde{\n} )} |\n - \tilde{\n}|^2 r^2 dr
+ \int_{( \n \ge 2 \tilde{\n} )} |\n - \tilde{\n}|^\ga r^2 dr \le C.
\ea\ee
This directly gives
\be\la{315}\ba
\int_0^\infty |\sqrt{\n} - \sqrt{\tilde{\n}}|^2 r^2 dr
& = \int_{( \n < 2 \tilde{\n} )} |\sqrt{\n} - \sqrt{\tilde{\n}}|^2 r^2 dr
+ \int_{( \n \ge 2 \tilde{\n} )} |\sqrt{\n} - \sqrt{\tilde{\n}}|^2 r^2 dr \\
& \le C \int_{( \n < 2 \tilde{\n} )} |\n - \tilde{\n}|^2 r^2 dr
+ C \int_{( \n \ge 2 \tilde{\n} )} |\n - \tilde{\n}|^\ga r^2 dr \\
& \le C,
\ea\ee
which implies
\be\la{316}\ba
\sup_{0 \le t \le T} \int_0^\infty |\sqrt{\n} - \sqrt{\tilde{\n}}|^2 r^2 dr \le C.
\ea\ee
Furthermore, multiplying $(\ref{yxsdfc1})$ by $w r^2$ and integrating by parts over $(0,\infty)$, we arrive at
\be\la{317}\ba
\frac{1}{2} \frac{d}{d t} \int_0^\infty \n w^2 r^2 dr
& = - \int_0^\infty \left( \p_r (\n^\ga) u r^2 + \p_r (\n^\ga) \n^{-1} \p_r(\n^\alpha) r^2 \right) dr \\
& = - \frac{d}{dt} \int_0^\infty K(\n) r^2 dr
- \frac{\ga}{\alpha} \int_0^\infty \n^{\ga-\alpha-1} |\p_r \n^\alpha|^2 r^2 dr.
\ea\ee
Integrating (\ref{317}) over $(0,T)$, we obtain
\be\la{318}\ba
\sup_{0 \le t \le T} \int_0^\infty \left( \frac{1}{2} \n w^2 + K(\n) \right) r^2 dr
+ \frac{\ga}{\alpha} \int_0^T \int_0^\infty \n^{\ga-\alpha-1} |\p_r \n^\alpha|^2 r^2 dr dt \le C.
\ea\ee
It follows from (\ref{313}), (\ref{318}), and Cauchy's inequality that
\be\la{319}\ba
\int_0^\infty |\p_r \n^{\alpha-\frac{1}{2}}|^2 r^2 dr
\le C \int_0^\infty \n \left( u^2 + w^2 \right) r^2 dr \le C.
\ea\ee
This, combined with the Sobolev inequality, yields
\be\la{3110}\ba
\| \n^{\alpha-\frac{1}{2}} - \tilde{\n}^{\alpha-\frac{1}{2}} \|_{L^6(\mathbb{R}^3)}^2
& \le C \| \na (\n^{\alpha-\frac{1}{2}} - \tilde{\n}^{\alpha-\frac{1}{2}}) \|_{L^2(\mathbb{R}^3)}^2 \\
& \le C \int_0^\infty |\p_r \n^{\alpha-\frac{1}{2}}|^2 r^2 dr \\
& \le C,
\ea\ee
which gives
\be\la{3111}\ba
\sup_{0 \le t \le T} \int_0^\infty |\n^{\alpha-\frac{1}{2}} - \tilde{\n}^{\alpha-\frac{1}{2}}|^6 r^2 dr \le C.
\ea\ee
The combination of (\ref{313}), (\ref{316}), (\ref{318}), (\ref{319}), and (\ref{3111}) implies (\ref{301}) and finishes the proof of Lemma \ref{3l1}.
\end{proof}

\begin{lemma}\la{3l2}
For any $\xi>0$, there exists a positive constant $C$ depending only on
$T$, $\xi$, $\underline{\n_0}$, $\overline{\n_0}$, $\| \n_0 - \tilde{\n} \|_{H^1}$, $\| \mathbf{u}_0 \|_{L^2}$, $\ga$, $\tilde{\n}$, and $\alpha$ such that
\be\la{302}\ba
\sup_{0 \le t \le T} \left( \| \n r^{\frac{1+\xi}{2\alpha-1}} \|_{L^\infty(0,1)} + \| \n \|_{L^\infty(1,\infty)} \right)
+ \int_0^T \| \n r^{\frac{1+\xi}{\alpha+\ga-1}} \|_{L^\infty(0,1)}^{\alpha+\ga-1} dt \le C.
\ea\ee
\end{lemma}
\begin{proof}
It follows from (\ref{ywsi2}) and H\"older's inequality that for any $\xi>0$,
\be\la{321}\ba
& \left\| \n^{s} r^{\frac{1+\xi}{2}} \right\|_{L^\infty(0,1)}^2 \\
& \le C \left( \int_0^1 \n^{s} r^{\frac{1+\xi}{2}} dr \right)^2
+ C \left( \int_0^1 \left( |\p_r \n^{s}| r^{\frac{1+\xi}{2}}
+ \n^{s} r^{\frac{\xi-1}{2}} \right) dr \right)^2 \\
& \le C \left( \int_0^1 \n^{s} r^{\frac{\xi-1}{2}} dr \right)^2
+ C \left( \int_0^1 |\p_r \n^{s}| r^{\frac{1+\xi}{2}} dr \right)^2 \\
& \le C \left( \int_0^1 \n^{6s} r^2 dr \right)^{\frac{1}{3}}
\left( \int_0^1 r^{\frac{3}{5}\xi-1} dr \right)^{\frac{5}{3}}
+ C \int_0^1 |\p_r \n^{s}|^2 r^2 dr \int_0^1 r^{\xi-1} dr \\
& \le C \left( \int_0^1 \n^{6s} r^2 dr \right)^{\frac{1}{3}} + C \int_0^1 |\p_r \n^s|^2 r^2 dr \\
& \le C \int_0^1 \n^{2s} r^2 dr + C \int_0^1 |\p_r \n^s|^2 r^2 dr.
\ea\ee
On the one hand, choosing $s=\alpha-\frac{1}{2}$ in (\ref{321}) and using (\ref{301}), we obtain
\be\la{322}\ba
\left\| \n r^{\frac{1+\xi}{2\alpha-1}} \right\|_{L^\infty(0,1)}^{2\alpha-1}
& = \left\| \n^{\alpha-\frac{1}{2}} r^{\frac{1+\xi}{2}} \right\|^2_{L^\infty(0,1)} \\
& \le C \int_0^1 \n^{2\alpha-1} r^2 dr + C \int_0^1 |\p_r \n^{\alpha-\frac{1}{2}}|^2 r^2 dr \\
& \le C + C \int_0^1 \chi_{(\n \ge 2 \tilde{\n})} \n^{2\alpha-1} r^2 dr \\
& \le C + C \int_0^1 \chi_{(\n \ge 2 \tilde{\n})} \n r^2 dr \\
& \le C + C \int_0^1 |\sqrt{\n} - \sqrt{\tilde{\n}}|^2 r^2 dr \\
& \le C,
\ea\ee
since $0<2\alpha-1<1$.

On the other hand, choosing $s=\frac{\alpha+\ga-1}{2}$ in (\ref{321}), integrating over $(0,T)$, and using (\ref{301}) and (\ref{314}), we arrive at
\be\la{323}\ba
& \int_0^T \| \n r^{\frac{1+\xi}{\alpha+\ga-1}} \|_{L^\infty(0,1)}^{\alpha+\ga-1} dt \\
& = \int_0^T \| \n^{ \frac{\alpha+\ga-1}{2} } r^{\frac{1+\xi}{2}} \|_{L^\infty(0,1)}^2 dt \\
& \le C \int_0^T \int_0^1 \left( \n^{\alpha+\ga-1} + |\p_r \n^\frac{\alpha+\ga-1}{2}|^2 \right) r^2 dr dt \\
& \le C + C \int_0^T \int_0^1 \chi_{(\n \ge 2 \tilde{\n})} \n^{\alpha+\ga-1} r^2 dr dt
+ C \int_0^1 \n^{\ga-\alpha-1} |\p_r \n^\alpha|^2 r^2 dr \\
& \le C + C \int_0^T \int_0^1 \chi_{(\n \ge 2 \tilde{\n})} |\n - \tilde{\n}|^\ga r^2 dr dt \\
& \le C,
\ea\ee
owing to $0<\alpha+\ga-1<\ga$.

Moreover, we use (\ref{301}), (\ref{ywsi1}), and H\"older's inequality to derive
\be\la{324}\ba
& \| (\n^{\alpha-\frac{1}{2}} - \tilde{\n}^{\alpha-\frac{1}{2}})^2 \|_{L^\infty(1,\infty)} \\
& \le \int_1^\infty | \p_r(\n^{\alpha-\frac{1}{2}} - \tilde{\n}^{\alpha-\frac{1}{2}})^2 | dr \\
& \le C \int_1^\infty | \n^{\alpha-\frac{1}{2}} - \tilde{\n}^{\alpha-\frac{1}{2}} | |\p_r \n^{\alpha-\frac{1}{2}}| dr \\
& \le C \left( \int_1^\infty | \n^{\alpha-\frac{1}{2}} - \tilde{\n}^{\alpha-\frac{1}{2}} |^6 r^2 dr \right)^{\frac{1}{6}}
\left( \int_1^\infty |\p_r \n^{\alpha-\frac{1}{2}}|^2 r^2 dr \right)^{\frac{1}{2}}
\left( \int_1^\infty r^{-4} dr \right)^{\frac{1}{3}} \\
& \le C.
\ea\ee
This implies that
\be\la{325}\ba
\| \n \|_{L^\infty(1,\infty)} = \| \n^{\alpha-\frac{1}{2}} \|^{\frac{2}{2\alpha-1}}_{L^\infty(1,\infty)}
\le \left( \| \n^{\alpha-\frac{1}{2}} - \tilde{\n}^{\alpha-\frac{1}{2}} \|_{L^\infty(1,\infty)} + \tilde{\n}^{\alpha-\frac{1}{2}} \right)^{\frac{2}{2\alpha-1}} \le C,
\ea\ee
which together with (\ref{322}) and (\ref{323}) yields (\ref{302}), thereby completing the proof of Lemma \ref{3l2}.
\end{proof}

To obtain the upper bound for $\n$, we require additional integrability of $\n u$ and $\n w$.
For $k \in (2,\infty)$, we set
\be\nonumber
g(k) \triangleq 1 - \frac{\sqrt{2k^3-k^2-2k+1}-3k+3}{(k-2)^2}.
\ee
It is easy to check that $g(k)$ is strictly increasing on $(2,\infty)$ and satisfies
\be\nonumber
\lim_{k \to 2^+} g(k) = \frac{2}{3}, \quad \lim_{k \to \infty} g(k) = 1.
\ee
Let $\psi(z) : (\frac{2}{3},1) \to (2,\infty)$ be the inverse function of $g(k)$.
Then $\psi(z)$ is strictly increasing on $(\frac{2}{3},1)$ and satisfies
\be\nonumber
\lim_{z \to \frac{2}{3}^+} \psi(z) = 2, \quad \lim_{z \to 1^-} \psi(z) = \infty.
\ee

Define
\be\la{rt}
R_T \triangleq 1 + \sup_{0 \le t \le T} \| \n \|_{L^\infty(0,1)}.
\ee

Next, we establish the additional integrability of the momentum $\n u$.

\begin{lemma}\la{3l3}
Assume that \eqref{th11} holds. For any $k \in [3,\psi(\alpha))$, there exists a constant $0 \le \si < (3\alpha-2)(1-\frac{2}{k})$ such that
\be\la{303}\ba
\sup_{0\le t \le T} \int_0^\infty \n |u|^{k} r^2 dr
+ \int_0^T \int_0^\infty \left( \n^\alpha |u|^{k} + \n^\alpha |\p_r u|^2 |u|^{k-2} r^2 \right) dr dt
\le C R_T^{k\si},
\ea\ee
where $C$ depends only on $T$, $k$, $\underline{\n_0}$, $\overline{\n_0}$, $\| \n_0 - \tilde{\n} \|_{H^1}$, $\| \mathbf{u}_0 \|_{H^2}$, $\ga$, $\tilde{\n}$, and $\alpha$.
\end{lemma}
\begin{proof}
For any $k \in [3,\psi(\alpha))$, multiplying $(\ref{nsqdc})_2$ by $|u|^{k-2}u r^2$ and integrating the resulting equation over $(0,\infty)$, we obtain after integration by parts that
\be\la{331}\ba
& \frac{1}{k} \frac{d}{dt} \int_0^\infty \n |u|^{k} r^2 dr
+ (4\alpha-2) \int_0^\infty \n^\alpha |u|^k dr
+ (k-1) \alpha \int_0^\infty \n^\alpha |u|^{k-2} |\p_r u|^2 r^2 dr \\
& = 2k (1-\alpha) \int_0^\infty \n^\alpha |u|^{k-2} u (\p_r u) r dr
+ (k-1) \int_0^\infty \n^\ga |u|^{k-2} (\p_r u) r^2 dr
+ 2 \int_0^\infty \n^{\ga} |u|^{k-2} u r dr.
\ea\ee
Since $3 \le k <\psi(\alpha)$, we have
\be\la{332}
1 - \frac{\sqrt{2k^3-k^2-2k+1}-3k+3}{(k-2)^2} < \alpha,
\ee
which implies
\be\la{333}
k^2 (1-\alpha)^2 < (4\alpha-2)(k-1)\alpha.
\ee
Hence, we can choose sufficiently small $\ep>0$ such that
\be\la{334}\ba
\frac{k^2(1-\alpha)^2}{(1-2\ep)(4\alpha-2)} \le (1-2\ep) (k-1) \alpha,
\ea\ee
which together with Young's inequality yields
\be\la{335}\ba
& 2k (1-\alpha) \n^\alpha |u|^{k-2} u (\p_r u) r \\
& \le (1-2\ep) (4\alpha-2) \n^\alpha |u|^k
+ \frac{k^2(1-\alpha)^2}{(1-2\ep)(4\alpha-2)} \n^{\alpha} |u|^{k-2} |\p_r u|^2 r^2 \\
& \le (1-2\ep) (4\alpha-2) \n^\alpha |u|^k
+ (1-2\ep) (k-1) \alpha \n^{\alpha} |u|^{k-2} |\p_r u|^2 r^2.
\ea\ee
Substituting (\ref{335}) into (\ref{331}) and using (\ref{302}) and Young's inequality, we derive
\be\la{336}\ba
& \frac{1}{k} \frac{d}{dt} \int_0^\infty \n |u|^{k} r^2 dr
+ \ep \int_0^\infty \n^\alpha |u|^k dr
+ 2 \ep \int_0^\infty \n^\alpha |u|^{k-2} |\p_r u|^2 r^2 dr \\
& \le (k-1) \int_0^\infty \n^\ga |u|^{k-2} (\p_r u) r^2 dr
+ 2 \int_0^\infty \n^{\ga} |u|^{k-2} u r dr \\
& = (k-1) \int_0^1 \n^\ga |u|^{k-2} (\p_r u) r^2 dr
+ 2 \int_0^1 \n^{\ga} |u|^{k-2} u r dr \\
& \quad + (k-1) \int_1^\infty \n^\ga |u|^{k-2} (\p_r u) r^2 dr
+ 2 \int_1^\infty \n^{\ga} |u|^{k-2} u r dr \\
& \le \frac{\ep}{2} \int_0^1 \n^\alpha |u|^k dr
+ \frac{\ep}{2} \int_0^1 \n^\alpha |u|^{k-2} |\p_r u|^2 r^2 dr
+ C \int_0^1 \n^{ k(\ga-\alpha)+\alpha } r^{k} dr \\
& \quad + (k-1) \int_1^\infty \n^\ga |u|^{k-2} (\p_r u) r^2 dr
+ C \int_1^\infty \n |u|^{2} r^2 dr + C \int_1^\infty \n |u|^{k} r^2 dr,
\ea\ee
which together with (\ref{301}) gives
\be\la{3361}\ba
& \frac{1}{k} \frac{d}{dt} \int_0^\infty \n |u|^{k} r^2 dr
+ \frac{\ep}{2} \int_0^\infty \n^\alpha |u|^k dr
+ \frac{3\ep}{2} \int_0^\infty \n^\alpha |u|^{k-2} |\p_r u|^2 r^2 dr \\
& \le C \int_0^1 \n^{ k(\ga-\alpha)+\alpha } r^{k} dr
+ (k-1) \int_1^\infty \n^\ga |u|^{k-2} (\p_r u) r^2 dr
+ C + C \int_1^\infty \n |u|^{k} r^2 dr.
\ea\ee
Note that, if $k \in [3,4]$, we use (\ref{301}), (\ref{302}) and the fact that $\alpha<1<\ga$ to obtain
\be\la{3362}\ba
& (k-1) \int_1^\infty \n^\ga |u|^{k-2} (\p_r u) r^2 dr \\
& \le C \int_1^\infty \n^{\ga} |u|^{2k-4} r^2 dr
+ C \int_1^\infty \n^{\ga} |\p_r u|^{2} r^2 dr \\
& \le C \int_1^\infty \n |u|^{2} r^2 dr
+ C \int_1^\infty \n |u|^{k} r^2 dr
+ C \int_1^\infty \n^\alpha |\p_r u|^{2} r^2 dr \\
& \le C + C \int_1^\infty \n |u|^{k} r^2 dr
+ C \int_1^\infty \n^\alpha |\p_r u|^{2} r^2 dr.
\ea\ee
On the other hand, if $k>4$, (\ref{301}), (\ref{302}), and Young's inequality ensure that
\be\la{3363}\ba
& (k-1) \int_1^\infty \n^\ga |u|^{k-2} (\p_r u) r^2 dr \\
& \le \ep \int_1^\infty \n^\alpha |u|^{k-2} |\p_r u|^2 r^2 dr
+ C \int_1^\infty \n^{2\ga-\alpha} |u|^{k-2} r^2 dr \\
& \le \ep \int_1^\infty \n^\alpha |u|^{k-2} |\p_r u|^2 r^2 dr
+ C \int_1^\infty \n |u|^{2} r^2 dr
+ C \int_1^\infty \n |u|^{k} r^2 dr \\
& \le \ep \int_1^\infty \n^\alpha |u|^{k-2} |\p_r u|^2 r^2 dr
+ C + C \int_1^\infty \n |u|^{k} r^2 dr.
\ea\ee
Combining (\ref{3361}), (\ref{3362}), and (\ref{3363}) leads to
\be\la{3364}\ba
& \frac{1}{k} \frac{d}{dt} \int_0^\infty \n |u|^{k} r^2 dr
+ \frac{\ep}{2} \int_0^\infty \n^\alpha |u|^k dr
+ \frac{\ep}{2} \int_0^\infty \n^\alpha |u|^{k-2} |\p_r u|^2 r^2 dr \\
& \le C \int_0^1 \n^{ k(\ga-\alpha)+\alpha } r^{k} dr
+ C \int_1^\infty \n^\alpha |\p_r u|^{2} r^2 dr
+ C \int_1^\infty \n |u|^{k} r^2 dr + C.
\ea\ee

We now estimate the first term in the last line of (\ref{3364}).
Since $1 < \ga < 6\alpha-3$, we can choose $0 \le \si < (3\alpha-2)(1-\frac{2}{k})$ such that
\be\la{337}\ba
(k-1)\ga < 3k\alpha-k+k\sigma-1.
\ea\ee
We distinguish the following two cases.

\vspace{0.5cm}

\textit{Case 1: $k\alpha+k\si-1 \le (k-1)\ga <3k\alpha-k+k\sigma-1$.}
In this case, we have
\be\la{338}
k - \frac{(k-1)\ga-k\alpha-k\si+1}{2\alpha-1} - 1 > -1.
\ee
Hence, we can choose $0<\xi_0<k-1$ sufficiently small such that
\be\la{339}
k-\left( \frac{(k-1)\ga-k\alpha-k\si+1}{2\alpha-1} + 1 \right) \left(1+\xi_0\right) > -1.
\ee
By virtue of (\ref{302}), (\ref{339}), and H\"older's inequality, it holds that
\be\la{3310}\ba
& \int_0^1 \n^{ k(\ga-\alpha)+\alpha } r^{k} dr \\
& \le C R_T^{k\si} \left\|\n r^{\frac{1+\xi_0}{2\alpha-1}} \right\|_{L^\infty(0,1)}^{(k-1)\ga-k\alpha-k\si+1}
\int_0^1 \n^{\alpha+\ga-1} r^{k-\left( \frac{(k-1)\ga-k\alpha-k\si+1}{2\alpha-1} \right) \left(1+\xi_0\right) } dr \\
& \le C R_T^{k\si} \left\|\n r^{\frac{1+\xi_0}{\alpha+\ga-1}} \right\|_{L^\infty(0,1)}^{\alpha+\ga-1}
\int_0^1 r^{k-\left( \frac{(k-1)\ga-k\alpha-k\si+1}{2\alpha-1} + 1 \right) \left(1+\xi_0\right) } dr \\
&\leq C R_T^{k\si} \left\|\n r^{\frac{1+\xi_0}{\alpha+\ga-1}} \right\|_{L^\infty(0,1)}^{\alpha+\ga-1}.
\ea\ee

\vspace{0.5cm}

\textit{Case 2: $(k-1)\ga < k\alpha+k\si-1$.}
It follows from (\ref{302}) and Young's inequality that
\be\la{3311}\ba
\int_0^1 \n^{ k(\ga-\alpha)+\alpha } r^{k} dr
& = \int_0^1 \n^{ (k-1)\ga + \ga - k\alpha+\alpha } r^{k} dr \\
& \le C + C \int_0^1 \n^{ k\si+\ga+\alpha-1 } r^{k} dr \\
& \le C + C R_T^{k\si}
\left\|\n r^{\frac{k}{\alpha+\ga-1}} \right\|_{L^\infty(0,1)}^{\alpha+\ga-1}.
\ea\ee
Combining (\ref{3364}), (\ref{3310}), and (\ref{3311}) and using the fact that $0<\xi_0<k-1$, we obtain
\be\la{3312}\ba
& \frac{1}{k} \frac{d}{dt} \int_0^\infty \n |u|^{k} r^2 dr
+ \frac{\ep}{2} \int_0^\infty \n^{\alpha} |u|^{k} dr
+ \ep \int_0^\infty \n^{\alpha} |\p_r u|^2  |u|^{k-2} r^2 dr \\
& \le C + C R_T^{k\si} \left\|\n r^{\frac{1+\xi_0}{\alpha+\ga-1}} \right\|_{L^\infty(0,1)}^{\alpha+\ga-1}
+ C \int_1^\infty \n^\alpha |\p_r u|^{2} r^2 dr
+ C \int_1^\infty \n |u|^{k} r^2 dr,
\ea\ee
which together with (\ref{301}), (\ref{302}), and Gr\"onwall's inequality yields
\be\la{3313}\ba
\sup_{0\le t \le T} \int_0^\infty \n |u|^{k} r^2 dr
+ \int_0^T \int_0^\infty \left( \n^\alpha |u|^{k} + \n^\alpha |\p_r u|^2 |u|^{k-2} r^2 \right) dr dt
\le C R_T^{k\si}.
\ea\ee
This gives (\ref{303}) and completes the proof of Lemma \ref{3l3}.
\end{proof}

\begin{lemma}\la{3l4}
Assume that \eqref{th11} holds. For $\si$ determined in Lemma \ref{3l3} and any $k \in [3,\psi(\alpha))$, there exists a positive constant $C$ depending only on
$T$, $k$, $\underline{\n_0}$, $\overline{\n_0}$, $\| \n_0 - \tilde{\n} \|_{H^3}$, $\| \mathbf{u}_0 \|_{H^2}$, $\ga$, $\tilde{\n}$, and $\alpha$ such that
\be\la{304}\ba
\sup_{0 \le t \le T} \int_0^\infty \n^{k(\alpha-2)+1} |\p_r\n|^{k} r^2 dr
+ \int_0^T \int_0^\infty \n^{k(\alpha-2)+\ga+1-\alpha} |\p_r\n|^{k} r^2 dr dt
\le C R_T^{k\si}.
\ea\ee
\end{lemma}
\begin{proof}
For any $k \in [3,\psi(\alpha))$, multiplying (\ref{yxsdfc1}) by $|w|^{k-2} w r^2$ and integrating by parts over $(0,\infty)$, we derive after using (\ref{yxsd}) and Young's inequality that
\be\la{341}\ba
& \frac{1}{k} \frac{d}{dt} \int_0^\infty \n |w|^{k} r^2 dr
+ \frac{\ga}{\alpha} \int_0^\infty \n^{\ga+1-\alpha} |w|^{k} r^2 dr \\
& = \frac{\ga}{\alpha} \int_0^\infty \n^{\ga+1-\alpha} |w|^{k-2} w u r^2 dr \\
& \le \frac{\ga}{2\alpha} \int_0^\infty \n^{\ga+1-\alpha} |w|^{k} r^2 dr
+ C \int_0^\infty \n^{\ga+1-\alpha} |u|^{k} r^2 dr.
\ea\ee
For $\xi_1 = \frac{6\alpha-3-\ga}{\ga+1-2\alpha}>0$, we deduce from (\ref{302}) and (\ref{303}) that
\be\la{342}\ba
& \int_0^\infty \n^{\ga+1-\alpha} |u|^{k} r^2 dr \\
& \le C \int_0^1 \n^{\ga+1-\alpha} |u|^{k} r^2 dr + C \int_1^\infty \n^{\ga+1-\alpha} |u|^{k} r^2 dr \\
& \le C \left\|\n r^{\frac{1+\xi_1}{2\alpha-1}} \right\|_{L^\infty(0,1)}^{\ga+1-2\alpha}
\int_0^1 \n^{\alpha} |u|^{k} dr
+ C \| \n \|_{L^\infty(1,\infty)}^{\ga-\alpha} \int_1^\infty \n |u|^{k} r^2 dr \\
& \le C \int_0^1 \n^{\alpha} |u|^{k} dr + C R_T^{k\si}.
\ea\ee
Integrating (\ref{342}) over $(0,T)$ and applying (\ref{303}), we arrive at
\be\la{343}\ba
\sup_{0 \le t \le T} \int_0^\infty \n |w|^{k} r^2 dr
+ \int_0^T \int_0^\infty \n^{\ga+1-\alpha} |w|^{k} r^2 dr dt
\le C R_T^{k\si}.
\ea\ee
This, combined with (\ref{yxsd}), (\ref{303}), (\ref{342}), and (\ref{343}), leads to
\be\la{344}\ba
\sup_{0 \le t \le T} \int_0^\infty \n^{1-k} |\p_r \n^\alpha|^k r^2 dr
+ \int_0^T \int_0^\infty \n^{\ga+1-\alpha-k} |\p_r \n^\alpha|^k r^2 dr dt
\le C R_T^{k\si},
\ea\ee
which yields (\ref{304}) and completes the proof of Lemma \ref{3l4}.
\end{proof}

\begin{lemma}\la{3l5}
Assume that \eqref{th11} holds.
Then there exists a positive constant $C$ depending only on
$T$, $\underline{\n_0}$, $\overline{\n_0}$, $\| \n_0 - \tilde{\n} \|_{H^3}$, $\| \mathbf{u}_0 \|_{H^2}$, $\ga$, $\tilde{\n}$, and $\alpha$ such that
\be\la{305}\ba
\sup_{0 \le t \le T} \| \n \|_{L^\infty(0,\infty)} \le C.
\ea\ee
\end{lemma}
\begin{proof}
For $\si$ determined in Lemma \ref{3l3}, which satisfies $0 \le \si < \left(3\alpha-2\right)\left(1-\frac{2}{k}\right)$, we can choose $\de$ such that
\be\la{351}\ba
\si < \de < \left(3\alpha-2\right)\left(1-\frac{2}{k}\right),
\ea\ee
which implies that
\be\la{352}\ba
-\frac{2}{k-1} - \left( \frac{k}{k-1}(\de-\alpha)+1 \right) \frac{1}{2\alpha-1} > -1.
\ea\ee
Thus, we can choose $\xi_2>0$ sufficiently small such that
\be\la{353}
-\frac{2}{k-1} - \left( \frac{k}{k-1}(\de-\alpha)+1 \right) \frac{1+\xi_2}{2\alpha-1} > -1.
\ee
By virtue of (\ref{301}), (\ref{302}), (\ref{304}), (\ref{353}), and the Sobolev embedding, we obtain
\be\la{354}\ba
\| \n^\de \|_{L^\infty(0,1)}
& \le C \int_0^1 \n^\de dr + C \int_0^1 |\p_r \n^\de| dr \\
& \le C \left( \int_0^1 \n^{6\alpha-3} r^2 dr \right)^{\frac{\de}{6\alpha-3}}
\left( \int_0^1 r^{-\frac{2\de}{6\alpha-3-\de}} dr \right)^{\frac{6\alpha-3-\de}{6\alpha-3}}
+ C \int_0^1 \n^{\de-1} |\p_r \n| dr \\
& \le C + C \int_0^1 |\n^{\alpha-\frac{1}{2}} - \tilde{\n}^{\alpha-\frac{1}{2}}|^6 r^2 dr \\
& \quad + C \left( \int_0^1 \n^{k(\alpha-2)+1} |\p_r\n|^{k} r^2 dr \right)^{\frac{1}{k}}
\left( \int_0^1 \n^{ \frac{k}{k-1}(\de-\alpha)+1 } r^{-\frac{2}{k-1}} dr \right)^{\frac{k-1}{k}} \\
& \le C + C R_T^\si \left( \left\| \n r^{\frac{1+\xi_2}{2\alpha-1}} \right\|_{L^\infty(0,1)}^{\frac{k}{k-1}(\de-\alpha)+1}
\int_0^1 r^{-\frac{2}{k-1} - \left( \frac{k}{k-1}(\de-\alpha)+1 \right) \frac{1+\xi_2}{2\alpha-1} } dr \right)^{\frac{k-1}{k}} \\
& \le C R_T^\si.
\ea\ee
From (\ref{354}) and (\ref{rt}), we conclude that
\be\la{355}
R_T^\de \le C R_T^\si,
\ee
which together with (\ref{351}) yields
\be\la{356}
\| \n \|_{L^\infty(0,1)} \le C.
\ee
This, combined with (\ref{302}), gives (\ref{305}) and finishes the proof of Lemma \ref{3l5}.
\end{proof}

Next, we derive the crucial lower bound for the density.
Note that, for $k_2>3$ given in (\ref{k2}), we have
\be\nonumber
1 - \frac{\sqrt{2k_2^3-k_2^2-2k_2+1}-3k_2+3}{(k_2-2)^2} = 1 - \frac{1}{k_2},
\ee
and for any $k>k_2$,
\be\nonumber
1 - \frac{\sqrt{2k^3-k^2-2k+1}-3k+3}{(k-2)^2} < 1 - \frac{1}{k}.
\ee
Thus, for any $1-\frac{1}{k_2}<\alpha<1$, there exists $k>k_2>3$ such that
\be\la{swk}
1 - \frac{\sqrt{2k^3-k^2-2k+1}-3k+3}{(k-2)^2} < \alpha < 1 - \frac{1}{k}.
\ee

\begin{lemma}\la{3l6}
Assume that \eqref{th11} holds.
Then there exists a positive constant $C$ depending only on
$T$, $\underline{\n_0}$, $\overline{\n_0}$, $\| \n_0 - \tilde{\n} \|_{H^3}$, $\| \mathbf{u}_0 \|_{H^2}$, $\ga$, $\tilde{\n}$, and $\alpha$ such that
\be\la{306}\ba
\sup_{0 \le t \le T} \| \n^{-1} \|_{L^\infty(0,\infty)} \le C.
\ea\ee
\end{lemma}
\begin{proof}
First, it follows from (\ref{301}) and (\ref{305}) that there exists a positive constant $\hat{C}_0$ depending only on $\n_0$, $\mathbf{u}_0$, $\ga$, $\tilde{\n}$, $\alpha$, and $T$ such that
\be\la{361}\ba
\sup_{0 \le t \le T} \int_0^\infty |\n - \tilde{\n}|^2 r^2 dr \le \hat{C}_0.
\ea\ee
For $k$ satisfying (\ref{swk}), by (\ref{ywsi1}), (\ref{301}), (\ref{304}), (\ref{305}), and H\"older's inequality, we derive
\be\la{362}\ba
& \| (\n^{-1} - \tilde{\n}^{-1})^2 \|_{L^\infty(1,\infty)} \\
& \le \int_1^\infty |\p_r (\n^{-1} - \tilde{\n}^{-1})^2| dr \\
& \le C \int_1^\infty \n^{-2} |\n^{-1} - \tilde{\n}^{-1}| |\p_r \n| dr \\
& \le C \int_1^\infty \chi_{(0 < \n < \frac{\tilde{\n}}{2})} \n^{-2} |\n^{-1} - \tilde{\n}^{-1}| |\p_r \n| dr
+ C \int_1^\infty \chi_{(\n \ge \frac{\tilde{\n}}{2})} \n^{-2} |\n^{-1} - \tilde{\n}^{-1}| |\p_r \n| dr \\
& \le C \int_1^\infty \chi_{(0 < \n < \frac{\tilde{\n}}{2})} (\n^{-1} - \tilde{\n}^{-1})^3 |\p_r \n| dr
+ C \int_1^\infty \chi_{(\n \ge \frac{\tilde{\n}}{2})} |\p_r \n^{\alpha-\frac{1}{2}}| dr \\
& \le C \| \n^{-1} - \tilde{\n}^{-1} \|_{L^\infty(1,\infty)}^{1+\alpha+\frac{1}{k}}
\int_1^\infty \chi_{(0 < \n < \frac{\tilde{\n}}{2})} (\n^{-1} - \tilde{\n}^{-1})^{2-\alpha-\frac{1}{k}} |\p_r \n| dr \\
& \quad + C \int_1^\infty \left( |\p_r \n^{\alpha-\frac{1}{2}}|^2 r^2 + r^{-2} \right) dr \\
& \le C \| \n^{-1} - \tilde{\n}^{-1} \|_{L^\infty(1,\infty)}^{1+\alpha+\frac{1}{k}}
\int_1^\infty \chi_{(0 < \n < \frac{\tilde{\n}}{2})} \n^{\alpha-2+\frac{1}{k}} |\p_r \n| dr
+ C \\
& \le C \| \n^{-1} - \tilde{\n}^{-1} \|_{L^\infty(1,\infty)}^{1+\alpha+\frac{1}{k}}
\left( \int_1^\infty \n^{k(\alpha-2)+1} |\p_r\n|^{k} dr \right)^{\frac{1}{k}}
\left( \int_1^\infty \chi_{(0 < \n < \frac{\tilde{\n}}{2})} dr \right)^{\frac{k-1}{k}}
+ C \\
& \le C \| \n^{-1} - \tilde{\n}^{-1} \|_{L^\infty(1,\infty)}^{1+\alpha+\frac{1}{k}}
+ C,
\ea\ee
where we have used the following estimate
\be\la{363}\ba
\int_1^\infty \chi_{(0 < \n < \frac{\tilde{\n}}{2})} dr
= \int_1^\infty \chi_{(\frac{\tilde{\n}}{2} < \tilde{\n} - \n < \tilde{\n})} dr
\le \left( \frac{\tilde{\n}}{2} \right)^{-2} \int_1^\infty |\tilde{\n} - \n|^2 r^2 dr
\le C,
\ea\ee
due to (\ref{361}).

Combining (\ref{362}) with Young's inequality and using the fact that $\alpha < 1 - \frac{1}{k}$, we arrive at
\be\la{364}\ba
\| \n^{-1} - \tilde{\n}^{-1} \|_{L^\infty(1,\infty)}^2
& = \| (\n^{-1} - \tilde{\n}^{-1})^2 \|_{L^\infty(1,\infty)} \\
& \le C \| \n^{-1} - \tilde{\n}^{-1} \|_{L^\infty(1,\infty)}^{1+\alpha+\frac{1}{k}}
+ C \\
& \le \frac{1}{2} \| \n^{-1} - \tilde{\n}^{-1} \|_{L^\infty(1,\infty)}^2 + C,
\ea\ee
which yields
\be\la{365}\ba
\| \n^{-1} - \tilde{\n}^{-1} \|_{L^\infty(1,\infty)}^2 \le C.
\ea\ee
This implies that
\be\la{3613}\ba
\| \n^{-1} \|_{L^\infty(1,\infty)}
& \le \| \n^{-1} - \tilde{\n}^{-1} \|_{L^\infty(1,\infty)} + \tilde{\n}^{-1}
\le C.
\ea\ee

It remains to estimate $\| \n^{-1} \|_{L^\infty(0,1)}$.
From (\ref{361}) and H\"older's inequality, we conclude that, for any $0 < r < \infty$,
\be\la{3615}\ba
\int_0^r \n s^2 ds
& = \int_0^r (\n - \tilde{\n}) s^2 ds + \tilde{\n} \int_0^r s^2 ds \\
& \ge - \left( \int_0^r |\n - \tilde{\n}|^2 s^2 ds \right)^{\frac{1}{2}} \left( \int_0^r s^2 ds \right)^{\frac{1}{2}}
+ \frac{1}{3} \tilde{\n} r^3 \\
& \ge - \hat{C}_0^{\frac{1}{2}} r^{\frac{3}{2}} + \frac{1}{3} \tilde{\n} r^3.
\ea\ee
We set
\be\la{3616}
\hat{r}_0 \triangleq 1 + \frac{3}{\tilde{\n}} (1 + \hat{C}_0^{\frac{1}{2}}),
\ee
and
\be\la{3617}\ba
\hat{y}_0 \triangleq \int_0^{\hat{r}_0} \n s^2 ds.
\ea\ee
Using (\ref{305}), (\ref{3615}), and (\ref{3616}), we arrive at
\be\la{3618}
1 \le \hat{y}_0 \le C.
\ee
For any $r\in(0,1)$ with $y$ defined by (\ref{lzb1}), we deduce from (\ref{ywsi2}), (\ref{lzb1}), (\ref{lzb2}), (\ref{304}), and (\ref{3613}) that
\be\la{3619}\ba
\n^{-1}(r,t) = \n^{-1}(y,t)
& \le \frac{1}{\hat{y}_0} \int_0^{\hat{y}_0} \n^{-1} dy
+ \int_0^{\hat{y}_0} |\p_y(\n^{-1})| dy \\
& \le \int_0^{\hat{r}_0} \n^{-1} \n r^2 dr + \int_0^{\hat{r}_0} |\p_r(\n^{-1})| dr \\
& \le \int_0^{\hat{r}_0} r^2 dr + C \int_0^{\hat{r}_0} \n^{-2} |\p_r \n| dr \\
& \le C + C \| \n^{-1} \|^{\alpha+\frac{1}{k}}_{L^\infty(0,\hat{r}_0)}
\left( \int_0^\infty \n^{k(\alpha-2)+1} |\p_r\n|^{k} r^2 dr \right)^{\frac{1}{k}}
\left( \int_0^{\hat{r}_0} r^{- \frac{2}{k-1}} dr \right)^{\frac{k-1}{k}} \\
& \le C + C \| \n^{-1} \|^{\alpha+\frac{1}{k}}_{L^\infty(0,\hat{r}_0)} \\
& \le C + C \| \n^{-1} \|^{\alpha+\frac{1}{k}}_{L^\infty(0,1)},
\ea\ee
due to $k>3$.

This, along with Young's inequality and the fact that $\alpha<1-\frac{1}{k}$, implies that
\be\la{3620}\ba
\sup_{0 \le t \le T} \| \n^{-1}(r,t) \|_{L^\infty(0,1)} \le C.
\ea\ee
Combining (\ref{3613}) and (\ref{3620}) yields (\ref{306}) and completes the proof of Lemma \ref{3l6}.
\end{proof}

\section{Higher Order Estimates}

In this subsection, we derive the higher-order estimates for $\OM=\mathbb{R}^N$ $(N=2 \text{ or } 3)$, which ensure that the classical solution can be extended globally in time.

\begin{lemma}\la{hl1}
There exists a positive constant $C$ depending only on $T$, $\alpha$, $\tilde{\n}$, $\ga$, $\underline{\n_0}$, $\| \n_0 - \tilde{\n} \|_{H^3}$, and $\| \mathbf{u}_0 \|_{H^2}$ such that
\be\la{h01}\ba
\sup_{0 \le t \le T} \left( \| \na \mathbf{u} \|^2_{L^2} + \| \n_t \|^2_{L^2} \right)
+ \int_0^T \left( \| \mathbf{u}_t \|_{L^2}^2 + \| \na^2 \mathbf{u} \|^2_{L^2} \right) dt \le C.
\ea\ee
\end{lemma}
\begin{proof}
First, it follows from (\ref{2101}), (\ref{203}), (\ref{204}), (\ref{205}), (\ref{206}), (\ref{301}), (\ref{303}), (\ref{304}), (\ref{305}), and (\ref{306}) that, for some $q \in (3,6]$,
\be\la{h11}\ba
\sup_{0 \le t \le T} \left( \| \mathbf{u} \|_{L^2} + \| \mathbf{u} \|_{L^q} + \| \n \|_{L^\infty} + \| \n^{-1} \|_{L^\infty} + \| \na \n \|_{L^2} + \| \na \n \|_{L^q} \right)
\le C.
\ea\ee
Since
\be\la{h12}\ba
|\na \mathbf{u}|^2 = |\p_r u|^2 + (N-1) \frac{1}{r^2} u^2,
\ea\ee
we obtain from (\ref{2101}) and (\ref{301}) that
\be\la{h13}\ba
\int_0^T \|\na \mathbf{u}\|_{L^2}^2 dt \le C.
\ea\ee
Using the fact that $\n>0$ and in the radially symmetric setting $\mathbb{D} \mathbf{u} = \na \mathbf{u}$, we deduce from $(\ref{ns})_2$ that $\mathbf{u}$ satisfies the following system
\be\la{h14}\ba
\begin{cases}
\Delta \mathbf{u} + (\alpha - 1) \na \div \mathbf{u}
= \n^{-\alpha} \left( \n \mathbf{u}_t + \n \mathbf{u} \cdot \na \mathbf{u} + \na P - \na \mu(\n) \cdot \na \mathbf{u} - \na \lambda(\n) \div \mathbf{u} \right), \\
\mathbf{u} \to 0 \ \text{ as } \ |x| \to \infty.
\end{cases}
\ea\ee
The standard $L^p$ estimate of elliptic equations (see \cite{NS}) implies that for $k \in \mathbb{N}$ and $p \in (1,\infty)$,
\be\la{h15}\ba
\| \na^2 \mathbf{u} \|_{W^{k,p}} \le C \| \n^{-\alpha} \left( \n \mathbf{u}_t + \n \mathbf{u} \cdot \na \mathbf{u} + \na P - \na \mu(\n) \cdot \na \mathbf{u} - \na \lambda(\n) \div \mathbf{u} \right) \|_{W^{k,p}}.
\ea\ee
In particular, by (\ref{gn11}), (\ref{h11}), and Young's inequality, we derive
\be\la{h16}\ba
\| \na^2 \mathbf{u} \|_{L^2}
& \le C \| \n^{-\alpha} \left( \n \mathbf{u}_t + \n \mathbf{u} \cdot \na \mathbf{u} + \na P - \na \mu(\n) \cdot \na \mathbf{u} - \na \lambda(\n) \div \mathbf{u} \right) \|_{L^2} \\
& \le C \left( \| \mathbf{u}_t \|_{L^2} + \| |\mathbf{u}| |\na \mathbf{u}| \|_{L^2} + \| \na \n \|_{L^2} + \| |\na \n| |\na \mathbf{u}| \|_{L^2} \right) \\
& \le C \left( 1 + \| \mathbf{u}_t \|_{L^2} + \| \mathbf{u} \|_{L^q} \| \na \mathbf{u} \|_{L^\frac{2q}{q-2}} + \| \na \n \|_{L^q} \| \na \mathbf{u} \|_{L^\frac{2q}{q-2}} \right) \\
& \le C \left( 1 + \| \mathbf{u}_t \|_{L^2}
+ \| \na \mathbf{u} \|_{L^2}^{\frac{q-N}{q}} \| \na^2 \mathbf{u} \|_{L^2}^{\frac{N}{q}} \right) \\
& \le \frac{1}{2} \| \na^2 \mathbf{u} \|_{L^2}
+ C \left( 1 + \| \mathbf{u}_t \|_{L^2} + \| \na \mathbf{u} \|_{L^2} \right),
\ea\ee
which yields
\be\la{h17}\ba
\| \na^2 \mathbf{u} \|_{L^2}
\le C \left( 1 + \| \mathbf{u}_t \|_{L^2} + \| \na \mathbf{u} \|_{L^2} \right).
\ea\ee
Multiplying $(\ref{h14})_1$ by $2 \mathbf{u}$, integrating by parts over $\mathbb{R}^N$, and using (\ref{h11}), (\ref{h16}), (\ref{gn11}), and Young's inequality, we arrive at
\be\la{h18}\ba
& \frac{d}{dt} \left( \| \na \mathbf{u} \|_{L^2}^2 + (\alpha-1) (\div \mathbf{u})^2 \right)
+ 2 \int \n^{1-\alpha} |\mathbf{u}_t|^2 dx \\
& \le C \int \left( |\mathbf{u}| |\na \mathbf{u}| + |\na \n| + |\na \n| |\na \mathbf{u}| \right) |\mathbf{u}_t| dx \\
& \le C \left( \| \mathbf{u} \|_{L^q} \| \na \mathbf{u} \|_{L^\frac{2q}{q-2}} + \| \na \n \|_{L^2}
+ \| \na \n \|_{L^q} \| \na \mathbf{u} \|_{L^\frac{2q}{q-2}} \right) \| \mathbf{u}_t \|_{L^2} \\
& \le C \left( 1 + \| \na \mathbf{u} \|_{L^2}^{\frac{q-N}{q}} \| \na^2 \mathbf{u} \|_{L^2}^{\frac{N}{q}} \right) \| \mathbf{u}_t \|_{L^2} \\
& \le C \left( 1 + \| \na \mathbf{u} \|_{L^2} + \| \na \mathbf{u} \|_{L^2}^{\frac{q-N}{q}} \| \mathbf{u}_t \|_{L^2}^{\frac{N}{q}} \right) \| \mathbf{u}_t \|_{L^2} \\
& \le \int \n^{1-\alpha} |\mathbf{u}_t|^2 dx + C
+ C \| \na \mathbf{u} \|_{L^2}^2,
\ea\ee
which gives
\be\la{h19}\ba
\frac{d}{dt} \left( \| \na \mathbf{u} \|_{L^2}^2 + (\alpha-1) (\div \mathbf{u})^2 \right)
+ \int \n^{1-\alpha} |\mathbf{u}_t|^2 dx
\le C + C \| \na \mathbf{u} \|_{L^2}^2.
\ea\ee
Integrating (\ref{h19}) over $(0,T)$ and using (\ref{h11}) and (\ref{h13}) lead to
\be\la{h110}\ba
\sup_{0 \le t \le T} \left( \| \na \mathbf{u} \|_{L^2}^2 + (\alpha-1) \| \div \mathbf{u} \|_{L^2}^2 \right)
+ \int_0^T\| \mathbf{u}_t \|_{L^2}^2 dt \le C.
\ea\ee
Note that for any vector-valued function $\mathbf{v}$, we have
\be\la{h111}\ba
(\div \mathbf{v})^2 \le N |\na \mathbf{v}|^2,
\ea\ee
which together with the fact that $\frac{N-1}{N} < \alpha <1$ implies
\be\la{h112}\ba
|\na \mathbf{v}|^2 + (\alpha-1) (\div \mathbf{v})^2
\ge (\alpha N - N + 1) |\na \mathbf{v}|^2 \triangleq N_{\alpha} |\na \mathbf{v}|^2,
\ea\ee
where $N_{\alpha} = \alpha N - N + 1 > 0$.
Hence, it holds that
\be\la{h1120}\ba
|\na \mathbf{u}|^2 + (\alpha-1) (\div \mathbf{u})^2
\ge N_{\alpha} |\na \mathbf{u}|^2.
\ea\ee
It follows from (\ref{h110}), (\ref{h1120}), and (\ref{h17}) that
\be\la{h113}\ba
\sup_{0 \le t \le T} \| \na \mathbf{u} \|_{L^2}^2
+ \int_0^T \left( \| \mathbf{u}_t \|_{L^2}^2 + \| \na^2 \mathbf{u} \|^2_{L^2} \right) dt \le C.
\ea\ee
Moreover, using $(\ref{ns})_1$, (\ref{h11}), (\ref{h113}), and H\"older's inequality, we obtain
\be\la{h114}\ba
\| \n_t \|_{L^2}
& \le C \left( \| \n \div \mathbf{u} \|_{L^2} + \| \mathbf{u} \cdot \na \n \|_{L^2} \right) \\
& \le C \left( \| \na \mathbf{u} \|_{L^2} + \| \mathbf{u} \|_{L^{\frac{2q}{q-2}}} \| \na \n \|_{L^q} \right) \\
& \le C.
\ea\ee
Combining (\ref{h113}) with (\ref{h114}) yields (\ref{h01}) and completes the proof of Lemma \ref{hl1}.
\end{proof}

\begin{lemma}\la{hl2}
There exists a positive constant $C$ depending only on $T$, $\alpha$, $\tilde{\n}$, $\ga$, $\underline{\n_0}$, $\| \n_0 - \tilde{\n} \|_{H^3}$, and $\| \mathbf{u}_0 \|_{H^2}$ such that
\be\la{h02}\ba
\sup_{0 \le t \le T} \left( \| \mathbf{u}_t \|_{L^2}^2 + \| \na^2 \mathbf{u} \|^2_{L^2} \right)
+ \int_0^T \| \na \mathbf{u}_t \|_{L^2}^2 dt \le C.
\ea\ee
\end{lemma}
\begin{proof}
First, applying $\p_t$ to $(\ref{ns})_2$ yields
\be\la{h21}\ba
& \n \mathbf{u}_{tt} + \n_t \mathbf{u}_t + \n \mathbf{u} \cdot \na \mathbf{u}_t
- \div(\n^\alpha \na \mathbf{u}_t)
- (\alpha-1) \na(\n^\alpha \div \mathbf{u}_t) \\
& = - \n_t \mathbf{u} \cdot \na \mathbf{u} - \n \mathbf{u}_t \cdot \na \mathbf{u}
+ \div( (\n^\alpha)_t \na \mathbf{u}) + (\alpha-1) \na((\n^\alpha)_t \div \mathbf{u}) - \na P_t.
\ea\ee
Multiplying (\ref{h21}) by $\mathbf{u}_t$, integrating over $\mathbb{R}^N$, and using (\ref{h01}), (\ref{h11}), (\ref{h17}), (\ref{gn11}), and Young's inequality, we derive
\be\la{h22}\ba
& \frac{1}{2} \frac{d}{dt} \int \n |\mathbf{u}_t|^2 dx + \int \n^\alpha ( |\na \mathbf{u}_t|^2 + (\alpha-1) (\div \mathbf{u}_t)^2 ) dx \\
& = - \int \n_t |\mathbf{u}_t|^2 dx
- \int \left( \n_t \mathbf{u} \cdot \na \mathbf{u} \cdot \mathbf{u}_t
- \n \mathbf{u}_t \cdot \na \mathbf{u} \cdot \mathbf{u}_t \right) dx \\
& \quad - \int (\n^\alpha)_t \na \mathbf{u} \cdot \na \mathbf{u}_t dx
- (\alpha-1) \int (\n^\alpha)_t \div \mathbf{u} \div \mathbf{u}_t dx
+ \int P_t \div \mathbf{u}_t dx \\
& \le C \| \n_t \|_{L^2} \| \mathbf{u}_t \|_{L^4}^2
+ C \| \n_t \|_{L^2} \| \mathbf{u} \|_{L^6} \| \na \mathbf{u} \|_{L^6} \| \mathbf{u}_t \|_{L^6}
+ C \| \na \mathbf{u} \|_{L^2} \| \mathbf{u}_t \|_{L^4}^2 \\
& \quad + C \| \n_t \|_{L^q} \| \na \mathbf{u} \|_{L^{\frac{2q}{q-2}}} \| \na \mathbf{u}_t \|_{L^2}
+ C \| \n_t \|_{L^2} \| \na \mathbf{u}_t \|_{L^2} \\
& \le C \| \mathbf{u}_t \|_{L^2}^{\frac{4-N}{2}} \| \na \mathbf{u}_t \|_{L^2}^{\frac{N}{2}}
+ C ( \| \mathbf{u}_t \|_{L^2} + \| \na \mathbf{u}_t \|_{L^2}) (1 + \| \na^2 \mathbf{u} \|_{L^2}) (1 + \| \mathbf{u}_t \|_{L^2}) \\
& \le C \| \mathbf{u}_t \|_{L^2}^{\frac{4-N}{2}} \| \na \mathbf{u}_t \|_{L^2}^{\frac{N}{2}}
+ C ( \| \mathbf{u}_t \|_{L^2} + \| \na \mathbf{u}_t \|_{L^2}) (1 + \| \mathbf{u}_t \|_{L^2}^2) \\
& \le \ep \int \n^\alpha |\na \mathbf{u}_t|^2 dx
+ C(\ep) + C(\ep) \| \sqrt{\n} \mathbf{u}_t \|_{L^2}^4,
\ea\ee
where we have used the following estimate
\be\la{h23}\ba
\| \n_t \|_{L^q}
& \le \left( \| \n \div \mathbf{u} \|_{L^q} + \| \mathbf{u} \cdot \na \n \|_{L^q} \right) \\
& \le C \left( \| \na \mathbf{u} \|_{H^1} + \| \mathbf{u} \|_{H^2} \| \na \n \|_{L^q} \right) \\
& \le C + C \| \mathbf{u}_t \|_{L^2},
\ea\ee
due to $(\ref{ns})_1$, (\ref{h01}), (\ref{h11}), and (\ref{h17}).

Choosing $\ep>0$ sufficiently small in (\ref{h22}) and using (\ref{h01}), (\ref{h11}), (\ref{h112}), and Gr\"onwall's inequality, we obtain
\be\la{h24}\ba
\sup_{0 \le t \le T} \| \mathbf{u}_t \|_{L^2}^2
+ \int_0^T \| \na \mathbf{u}_t \|_{L^2}^2 dt \le C.
\ea\ee
Furthermore, we conclude from (\ref{h24}), (\ref{h23}), (\ref{h01}), and (\ref{h17}) that
\be\la{h25}\ba
\sup_{0 \le t \le T} \left( \| \na^2 \mathbf{u} \|^2_{L^2} + \| \n_t \|_{L^q} \right) \le C,
\ea\ee
which together with (\ref{h24}) gives (\ref{h02}) and finishes the proof of Lemma \ref{hl2}.
\end{proof}

\begin{lemma}\la{hl3}
There exists a positive constant $C$ depending only on $T$, $\alpha$, $\tilde{\n}$, $\ga$, $\underline{\n_0}$, $\| \n_0 - \tilde{\n} \|_{H^3}$, and $\| \mathbf{u}_0 \|_{H^2}$ such that
\be\la{h03}\ba
\sup_{0 \le t \le T} \left( \| \na^2 \n \|_{L^2}^2 + \| \na \n_t \|_{L^2}^2 \right)
+ \int_0^T \| \na^3 \mathbf{u} \|_{L^2}^2 dt \le C.
\ea\ee
\end{lemma}
\begin{proof}
First, applying $\na^2$ to $(\ref{ns})_1$, multiplying the resulting equation by $\na^2 \n$, integrating by parts over $\mathbb{R}^N$, and using (\ref{h11}), (\ref{h02}), and the Sobolev inequality, we obtain
\be\la{h32}\ba
& \frac{d}{dt} \int |\na^2 \n|^2 dx \\
& \le C \int \left( |\na^2 \mathbf{u}| |\na \n| + |\na^2 \n| |\na \mathbf{u}| + |\na^3 \mathbf{u}| \right) |\na^2 \n| dx \\
& \le C \| \na^2 \mathbf{u} \|_{L^q} \| \na \n \|_{L^{\frac{2q}{q-2}}} \| \na^2 \n \|_{L^2}
+ C \| \na \mathbf{u} \|_{L^\infty} \| \na^2 \n \|_{L^2}^2
+ C \| \na^3 \mathbf{u} \|_{L^2} \| \na^2 \n \|_{L^2} \\
& \le C \left( 1 + \| \na^2 \mathbf{u} \|_{L^q} \right)
\left( 1 + \| \na^2 \n \|_{L^2}^2 \right)
+ C \| \na^3 \mathbf{u} \|_{L^2} \| \na^2 \n \|_{L^2}.
\ea\ee
It follows from (\ref{gn11}), (\ref{h15}), (\ref{h01}), (\ref{h11}), (\ref{h02}), and Young's inequality that
\be\la{h33}\ba
\| \na^2 \mathbf{u} \|_{L^q}
& \le C \| \n^{-\alpha} \left( \n \mathbf{u}_t + \n \mathbf{u} \cdot \na \mathbf{u} + \na P - \na \mu(\n) \cdot \na \mathbf{u} - \na \lambda(\n) \div \mathbf{u} \right) \|_{L^{q}} \\
& \le C \left( \| \mathbf{u}_t \|_{L^q} + \| \mathbf{u} \|_{L^q} \| \na \mathbf{u} \|_{L^\infty} + \| \na \n \|_{L^q} + \| \na \n \|_{L^q} \| \na \mathbf{u} \|_{L^\infty} \right) \\
& \le C \left( 1 + \| \na \mathbf{u}_t \|_{L^2}
+ \| \na \mathbf{u} \|_{L^2}^{\frac{2(q-N)}{N(q-2)+2q}} \| \na^2 \mathbf{u} \|_{L^q}^{\frac{qN}{N(q-2)+2q}} \right) \\
& \le \frac{1}{2} \| \na^2 \mathbf{u} \|_{L^q} + C \left( 1 + \| \na \mathbf{u}_t \|_{L^2} \right),
\ea\ee
which gives
\be\la{h34}\ba
\| \na^2 \mathbf{u} \|_{L^q}
\le C \left( 1 + \| \na \mathbf{u}_t \|_{L^2} \right).
\ea\ee
By virtue of (\ref{h01}), (\ref{h11}), (\ref{h15}), (\ref{h02}), and (\ref{h34}), we have
\be\la{h35}\ba
& \| \na^3 \mathbf{u} \|_{L^2} \\
& \le C \| \n^{-\alpha} \left( \n \mathbf{u}_t + \n \mathbf{u} \cdot \na \mathbf{u} + \na P - \na \mu(\n) \cdot \na \mathbf{u} - \na \lambda(\n) \div \mathbf{u} \right) \|_{H^1} \\
& \le C \left( \| \mathbf{u}_t \|_{L^2} + \| \mathbf{u} \|_{L^4} \| \na \mathbf{u} \|_{L^4}
+ \| \na \n \|_{L^2} + \| \na \n \|_{L^4} \| \na \mathbf{u} \|_{L^4} \right) \\
& \quad + C \left( \| \na \n \|_{L^4} \| \mathbf{u} \|_{L^4} + \| \na \mathbf{u}_t \|_{L^2} + \| \na \n \|_{L^4} \| \mathbf{u} \|_{L^\infty} \| \na \mathbf{u} \|_{L^4} + \| \na \mathbf{u} \|^2_{L^4} + \| \mathbf{u} \|_{L^\infty} \| \na^2 \mathbf{u} \|_{L^2} \right) \\
& \quad + C \left( \| \na \n \|^2_{L^4} + \| \na^2 \n \|_{L^2} + \| \na \n \|_{L^4}^2 \| \na \mathbf{u} \|_{L^\infty} + \| \na^2 \n \|_{L^2} \| \na \mathbf{u} \|_{L^\infty} + \| \na \n \|_{L^{\frac{2q}{q-2}}} \| \na^2 \mathbf{u} \|_{L^q} \right) \\
& \le C \left( 1 + \| \na \n \|^2_{L^4} + \| \na^2 \n \|_{L^2} \right) \left( 1 + \| \na \mathbf{u}_t \|_{L^2} \right) \\
& \le C \left( 1 + \| \na \n \|^{\frac{q(N-4)}{2q-N(q-2)}}_{L^q} \| \na^2 \n \|^{\frac{N(4-q)}{2q-N(q-2)}}_{L^2} + \| \na^2 \n \|_{L^2} \right) \left( 1 + \| \na \mathbf{u}_t \|_{L^2} \right) \\
& \le C \left( 1 + \| \na^2 \n \|_{L^2} \right) \left( 1 + \| \na \mathbf{u}_t \|_{L^2} \right).
\ea\ee
Combining (\ref{h32}), (\ref{h34}), and (\ref{h35}) leads to
\be\la{h36}\ba
\frac{d}{dt} \int |\na^2 \n|^2 dx
\le C \left( 1 + \| \na^2 \n \|^2_{L^2} \right) \left( 1 + \| \na \mathbf{u}_t \|_{L^2}^2 \right),
\ea\ee
which together with (\ref{h02}) and Gr\"onwall's inequality implies
\be\la{h37}\ba
\sup_{0 \le t \le T} \| \na^2 \n \|_{L^2}^2 \le C.
\ea\ee
This, along with (\ref{h35}) and (\ref{h02}), yields
\be\la{h38}\ba
\int_0^T \| \na^3 \mathbf{u} \|_{L^2}^2 dt \le C.
\ea\ee
Moreover, from $(\ref{ns})_1$, we deduce that $\na \n$ satisfies
\be\la{h39}\ba
\na \n_t + \na \mathbf{u} \cdot \na \n + \mathbf{u} \cdot \na^2 \n + \na \n \div \mathbf{u} + \n \na \div \mathbf{u} = 0.
\ea\ee
Using (\ref{h39}), (\ref{h37}), (\ref{h01}), (\ref{h11}), and (\ref{h02}), we arrive at
\be\la{h310}\ba
\| \na \n_t \|_{L^2}
\le C \left( \| \na \mathbf{u} \|_{L^4} \| \na \n \|_{L^4} + \| \mathbf{u} \|_{L^\infty} \| \na^2 \n \|_{L^2} + \| \na^2 \mathbf{u} \|_{L^2} \right)
\le C.
\ea\ee
The combination of (\ref{h37}), (\ref{h38}), and (\ref{h310}) gives (\ref{h03}) and completes the proof of Lemma \ref{hl3}.
\end{proof}

\begin{lemma}\la{hl4}
There exists a positive constant $C$ depending only on $T$, $\alpha$, $\tilde{\n}$, $\ga$, $\underline{\n_0}$, $\| \n_0 - \tilde{\n} \|_{H^3}$, and $\| \mathbf{u}_0 \|_{H^3}$ such that
\be\la{h04}\ba
\sup_{0 \le t \le T} \left( \| \na \mathbf{u}_t \|_{L^2}^2 + \| \na^3 \mathbf{u} \|_{L^2}^2 + \| \n_{tt} \|^2_{L^2} \right)
+ \int_0^T \left( \| \mathbf{u}_{tt} \|_{L^2}^2 + \| \na^2 \mathbf{u}_t \|_{L^2}^2 \right) dt \le C.
\ea\ee
\end{lemma}
\begin{proof}
First, applying $\p_t$ to (\ref{h14}) gives
\be\la{h41}\ba
& \n^{1-\alpha} \mathbf{u}_{tt} - \Delta \mathbf{u}_t - (\alpha - 1) \na \div \mathbf{u}_t \\
& = - (\n^{1-\alpha})_t \mathbf{u}_t
- \left( \n^{-\alpha} \left( \n \mathbf{u} \cdot \na \mathbf{u} + \na P - \na \mu(\n) \cdot \na \mathbf{u} - \na \lambda(\n) \div \mathbf{u} \right) \right)_t.
\ea\ee
Multiplying (\ref{h41}) by $\mathbf{u}_{tt}$, integrating by parts over $\mathbb{R}^N$, and using (\ref{h01}), (\ref{h11}), (\ref{h02}), (\ref{h03}), and (\ref{h34}), we derive
\be\la{h42}\ba
& \frac{1}{2} \frac{d}{dt} \int \left(|\na \mathbf{u}_t|^2 + (\alpha-1) (\div \mathbf{u}_t)^2 \right) dx
+ \int \n^{1-\alpha} |\mathbf{u}_{tt}|^2 dx \\
& \le C \int \left( |\n_t| |\mathbf{u}_t| + |\n_t| |\mathbf{u}| |\na \mathbf{u}| + |\mathbf{u}_t| |\na \mathbf{u}| + |\mathbf{u}| |\na \mathbf{u}_t| \right) |\mathbf{u}_{tt}| dx \\
& \quad + C \int \left( |\n_t| |\na \n| + |\na \n_t| + |\n_t| |\na \n| |\na \mathbf{u}| + |\na \n_t| |\na \mathbf{u}| + |\na \n| |\na \mathbf{u}_t| \right) |\mathbf{u}_{tt}| dx \\
& \le C \left( \| \n_t \|_{L^4} \| \mathbf{u}_t \|_{L^4} + \| \n_t \|_{L^4} \| \mathbf{u} \|_{L^\infty} \| \na \mathbf{u} \|_{L^4} + \| \mathbf{u}_t \|_{L^4} \| \na \mathbf{u} \|_{L^4} \right) \| \mathbf{u}_{tt} \|_{L^2} \\
& \quad + C \left( \| \mathbf{u} \|_{L^\infty} \| \na \mathbf{u}_t \|_{L^2} + \| \n_t \|_{L^4} \| \na \n \|_{L^4} + \| \na \n_t \|_{L^2} \right) \| \mathbf{u}_{tt} \|_{L^2} \\
& \quad + C \left( \| \n_t \|_{L^4} \| \na \n \|_{L^4} \| \na \mathbf{u} \|_{L^\infty} + \| \na \n_t \|_{L^2} \| \na \mathbf{u} \|_{L^\infty} + \| \na \n \|_{L^4} \| \na \mathbf{u}_t \|_{L^4} \right) \| \mathbf{u}_{tt} \|_{L^2} \\
& \le C \left( 1 + \| \na \mathbf{u}_t \|_{L^2} + \| \na \mathbf{u}_t \|_{L^4} \right) \| \mathbf{u}_{tt} \|_{L^2}.
\ea\ee
In addition, applying standard elliptic estimates to (\ref{h41}) and using (\ref{h01}), (\ref{h11}), (\ref{h02}), (\ref{h03}), (\ref{h34}), (\ref{gn11}), and Young's inequality, we obtain
\be\la{h43}\ba
\| \na^2 \mathbf{u}_t \|_{L^2}
& \le C \| \left( \n^{-\alpha} \left( \n \mathbf{u}_t + \n \mathbf{u} \cdot \na \mathbf{u} + \na P - \na \mu(\n) \cdot \na \mathbf{u} - \na \lambda(\n) \div \mathbf{u} \right) \right)_t \|_{L^2} \\
& \le C \left( \| \n_t \|_{L^4} \| \mathbf{u}_t \|_{L^4} + \| \mathbf{u}_{tt} \|_{L^2} + \| \n_t \|_{L^4} \| \mathbf{u} \|_{L^\infty} \| \na \mathbf{u} \|_{L^4} \right) \\
& \quad + C \left( \| \mathbf{u}_t \|_{L^4} \| \na \mathbf{u} \|_{L^4} + \| \mathbf{u} \|_{L^\infty} \| \na \mathbf{u}_t \|_{L^2} + \| \n_t \|_{L^4} \| \na \n \|_{L^4} + \| \na \n_t \|_{L^2} \right) \\
& \quad + C \left( \| \n_t \|_{L^4} \| \na \n \|_{L^4} \| \na \mathbf{u} \|_{L^\infty} + \| \na \n_t \|_{L^2} \| \na \mathbf{u} \|_{L^\infty} + \| \na \n \|_{L^4} \| \na \mathbf{u}_t \|_{L^4} \right) \\
& \le C \left( 1 + \| \mathbf{u}_{tt} \|_{L^2} + \| \na \mathbf{u}_t \|_{L^2} + \| \na \mathbf{u}_t \|_{L^2}^{\frac{4-N}{4}} \| \na^2 \mathbf{u}_t \|_{L^2}^{\frac{N}{4}} \right) \\
& \le \frac{1}{2} \| \na^2 \mathbf{u}_t \|_{L^2}
+ C \left( 1 + \| \mathbf{u}_{tt} \|_{L^2} + \| \na \mathbf{u}_t \|_{L^2} \right),
\ea\ee
which implies
\be\la{h44}\ba
\| \na^2 \mathbf{u}_t \|_{L^2}
\le C \left( 1 + \| \mathbf{u}_{tt} \|_{L^2} + \| \na \mathbf{u}_t \|_{L^2} \right).
\ea\ee
It follows from (\ref{h42}), (\ref{h44}), (\ref{h11}), (\ref{gn11}), and Young's inequality that
\be\la{h45}\ba
& \frac{1}{2} \frac{d}{dt} \int \left(|\na \mathbf{u}_t|^2 + (\alpha-1) (\div \mathbf{u}_t)^2 \right) dx
+ \int \n^{1-\alpha} |\mathbf{u}_{tt}|^2 dx \\
& \le C \left( 1 + \| \na \mathbf{u}_t \|_{L^2} + \| \na \mathbf{u}_t \|_{L^2}^{\frac{4-N}{4}} \| \na^2 \mathbf{u}_t \|_{L^2}^{\frac{N}{4}} \right) \| \mathbf{u}_{tt} \|_{L^2} \\
& \le C \left( 1 + \| \na \mathbf{u}_t \|_{L^2} + \| \na \mathbf{u}_t \|_{L^2}^{\frac{4-N}{4}} \| \mathbf{u}_{tt} \|_{L^2}^{\frac{N}{4}} \right) \| \mathbf{u}_{tt} \|_{L^2} \\
& \le \frac{1}{2} \int \n^{1-\alpha} |\mathbf{u}_{tt}|^2 dx
+ C \left( 1 + \| \na \mathbf{u}_t \|^2_{L^2} \right).
\ea\ee
Integrating (\ref{h45}) over $(0,T)$ and using (\ref{h02}) and (\ref{h112}), we arrive at
\be\la{h46}\ba
\sup_{0 \le t \le T} \| \na \mathbf{u}_t \|_{L^2}^2
+ \int_0^T \| \mathbf{u}_{tt} \|_{L^2}^2 dt \le C,
\ea\ee
which together with (\ref{h03}), (\ref{h35}), and (\ref{h44}) yields
\be\la{h47}\ba
\sup_{0 \le t \le T} \| \na^3 \mathbf{u} \|_{L^2}^2
+ \int_0^T \| \na^2 \mathbf{u}_t \|_{L^2}^2 dt \le C.
\ea\ee
Moreover, from $(\ref{ns})_1$, we deduce that
\be\la{h48}\ba
|\n_{tt}| \le C \left( |\mathbf{u}_t| |\na \n| + |\mathbf{u}| |\na \n_t| + |\n_t| |\na \mathbf{u}| + \n |\na \mathbf{u}_t| \right).
\ea\ee
Using (\ref{h01}), (\ref{h02}), (\ref{h03}), (\ref{h46}), (\ref{h48}), and H\"older's inequality, we obtain
\be\la{h49}\ba
\| \n_{tt} \|_{L^2}
& \le C \left( \| \mathbf{u}_t \|_{L^4} \| \na \n \|_{L^4} + \| \mathbf{u} \|_{L^\infty} \| \na \n_t \|_{L^2}
+ \| \n_t \|_{L^4} \| \na \mathbf{u} \|_{L^4} + \| \na \mathbf{u}_t \|_{L^2} \right) \\
& \le C \left( 1 + \| \na \mathbf{u}_t \|_{L^2}  \right) \le C.
\ea\ee
Combining (\ref{h46}), (\ref{h47}), and (\ref{h49}) gives (\ref{h04}), thereby completing the proof of Lemma \ref{hl4}.
\end{proof}

\begin{lemma}\la{hl5}
There exists a positive constant $C$ depending only on $T$, $\alpha$, $\tilde{\n}$, $\ga$, $\underline{\n_0}$, $\| \n_0 - \tilde{\n} \|_{H^3}$, and $\| \mathbf{u}_0 \|_{H^3}$ such that
\be\la{h05}\ba
\sup_{0 \le t \le T} \left( \| \na^3 \n \|_{L^2}^2 + \| \na^2 \n_t \|_{L^2}^2 \right)
+ \int_0^T \| \na^4 \mathbf{u} \|_{L^2}^2 dt \le C.
\ea\ee
\end{lemma}
\begin{proof}
First, applying $\na^3$ to $(\ref{ns})_1$, multiplying the resulting equation by $\na^3 \n$, integrating by parts over $\mathbb{R}^N$, and using (\ref{h01}), (\ref{h11}), (\ref{h02}), (\ref{h03}), and (\ref{h04}), we arrive at
\be\la{h51}\ba
& \frac{d}{dt} \int |\na^3 \n|^2 dx \\
& \le C \int \left( |\na \mathbf{u}| |\na^3 \n| + |\na^2 \mathbf{u}| |\na^2 \n| + |\na^3 \mathbf{u}| |\na \n| + |\na^4 \mathbf{u}| \right) |\na^3 \n| dx \\
& \le C \left( \| \na \mathbf{u} \|_{L^\infty} \| \na^3 \n \|_{L^2}
+ \| \na^2 \mathbf{u} \|_{L^4} \| \na^2 \n \|_{L^4}
+ \| \na^3 \mathbf{u} \|_{L^2} \| \na \n \|_{L^\infty} + \| \na^4 \mathbf{u} \|_{L^2} \right) \| \na^3 \n \|_{L^2} \\
& \le C \left( 1 + \| \na^3 \n \|_{L^2} + \| \na^4 \mathbf{u} \|_{L^2} \right) \| \na^3 \n \|_{L^2}.
\ea\ee
Moreover, by virtue of (\ref{h15}), (\ref{h35}), (\ref{h01}), (\ref{h11}), (\ref{h02}), (\ref{h03}), and (\ref{h04}), we derive
\be\la{h52}\ba
& \| \na^4 \mathbf{u} \|_{L^2} \\
& \le C \| \n^{-\alpha} \left( \n \mathbf{u}_t + \n \mathbf{u} \cdot \na \mathbf{u} + \na P - \na \mu(\n) \cdot \na \mathbf{u} - \na \lambda(\n) \div \mathbf{u} \right) \|_{H^2} \\
& \le C + C \left( \| |\na \n|^2 |\mathbf{u}_t| \|_{L^2} + \| |\na^2 \n| |\mathbf{u}_t| \|_{L^2} + \| |\na \n| |\na \mathbf{u}_t| \|_{L^2} + \| \na^2 \mathbf{u}_t \|_{L^2} \right) \\
& \quad + C \left( \| |\na \n|^2 |\mathbf{u}| |\na \mathbf{u}| \|_{L^2} + \| |\na^2 \n| |\mathbf{u}| |\na \mathbf{u}| \|_{L^2} + \| |\na \n| |\na \mathbf{u}|^2 \|_{L^2} + \| |\na \n| |\mathbf{u}| |\na^2 \mathbf{u}| \|_{L^2} \right) \\
& \quad + C \left( \| |\na \mathbf{u}| |\na^2 \mathbf{u}| \|_{L^2} + \| |\mathbf{u}| |\na^3 \mathbf{u}| \|_{L^2} + \| |\na \n|^3 \|_{L^2} + \| |\na \n| |\na^2 \n| \|_{L^2} + \| \na^3 \n \|_{L^2} \right) \\
& \quad + C \left( \| |\na \n|^3 |\na \mathbf{u}| \|_{L^2} + \| |\na \n| |\na^2 \n| |\na \mathbf{u}| \|_{L^2} + \| |\na \n|^2 |\na^2 \mathbf{u}| \|_{L^2} + \| |\na^3 \n| |\na \mathbf{u}| \|_{L^2} \right) \\
& \quad + C \left( \| |\na^2 \n| |\na^2 \mathbf{u}| \|_{L^2} + \| |\na \n| |\na^3 \mathbf{u}| \|_{L^2} \right) \\
& \le C \left( 1 + \| \na^2 \mathbf{u}_t \|_{L^2} + \| \na^3 \n \|_{L^2} \right).
\ea\ee
Putting (\ref{h52}) into (\ref{h51}) leads to
\be\la{h53}\ba
\frac{d}{dt} \int |\na^3 \n|^2 dx
& \le C \left( 1 + \| \na^3 \n \|_{L^2} + \| \na^2 \mathbf{u}_t \|_{L^2} \right) \| \na^3 \n \|_{L^2} \\
& \le C \| \na^3 \n \|^2_{L^2} + C \left( 1 + \| \na^2 \mathbf{u}_t \|^2_{L^2} \right),
\ea\ee
which together with (\ref{h04}) and Gr\"onwall's inequality yields
\be\la{h54}\ba
\sup_{0 \le t \le T} \| \na^3 \n \|_{L^2}^2 \le C.
\ea\ee
This, combined with (\ref{h52}) and (\ref{h04}), shows that
\be\la{h55}\ba
\int_0^T \| \na^4 \mathbf{u} \|_{L^2}^2 dt \le C.
\ea\ee
Finally, by direct calculation from $(\ref{ns})_1$, we obtain
\be\la{h56}\ba
|\na^2 \n_t| \le C \left( |\na^2 \mathbf{u}| |\na \n| + |\na \mathbf{u}| |\na^2 \n| + |\mathbf{u}| |\na^3 \n| + \n |\na^3 \mathbf{u}| \right).
\ea\ee
By virtue of (\ref{h54}), (\ref{h01}), (\ref{h02}), (\ref{h03}), and (\ref{h04}), it holds that
\be\la{h57}\ba
\| \na^2 \n_t \|_{L^2}
& \le C \left( \| \na^2 \mathbf{u} \|_{L^4} \| \na \n \|_{L^4}
+ \| \na \mathbf{u} \|_{L^4} \| \na^2 \n \|_{L^4} + \| \mathbf{u} \|_{L^\infty} \| \na^3 \n \|_{L^2} + \| \na^3 \mathbf{u} \|_{L^2} \right) \\
& \le C.
\ea\ee
The combination of (\ref{h54}), (\ref{h55}), and (\ref{h57}) gives (\ref{h05}) and completes the proof of Lemma \ref{hl5}.
\end{proof}

\begin{lemma}\la{hl6}
There exists a positive constant $C$ depending only on $T$, $\alpha$, $\tilde{\n}$, $\ga$, $\underline{\n_0}$, $\| \n_0 - \tilde{\n} \|_{H^3}$, and $\| \mathbf{u}_0 \|_{H^3}$ such that
\be\la{h06}\ba
\sup_{0 \le t \le T} t \left( \| \mathbf{u}_{tt} \|_{L^2}^2 + \| \na^2 \mathbf{u}_t \|_{L^2}^2 + \| \na^4 \mathbf{u} \|^2_{L^2} \right)
+ \int_0^T t \| \na \mathbf{u}_{tt} \|_{L^2}^2 dt \le C.
\ea\ee
\end{lemma}
\begin{proof}
First, differentiating $(\ref{ns})_2$ with respect to $t$ yields
\be\la{h61}\ba
& \n \mathbf{u}_{ttt} + \n \mathbf{u} \cdot \na \mathbf{u}_{tt}
- \div(\n^\alpha \na \mathbf{u}_{tt})
- (\alpha-1) \na(\n^\alpha \div \mathbf{u}_{tt}) \\
& = - \n_{tt} \mathbf{u}_t - 2 \n_t \mathbf{u}_{tt}
-(\n_{tt} \mathbf{u} + 2\n_t \mathbf{u}_t + \n \mathbf{u}_{tt}) \cdot \na \mathbf{u}
- 2 (\n \mathbf{u}_t + \n_t \mathbf{u}) \cdot \na \mathbf{u}_t + \div((\n^\alpha)_{tt} \na \mathbf{u}) \\
& \quad + 2 \div((\n^\alpha)_t \na \mathbf{u}_t)
+ (\alpha-1) \na((\n^\alpha)_{tt} \div \mathbf{u}) + 2 (\alpha-1) \na((\n^\alpha)_t \div \mathbf{u}_t)
- \na P_{tt}.
\ea\ee
Multiplying (\ref{h61}) by $\mathbf{u}_{tt}$, integrating by parts over $\mathbb{R}^N$, and using (\ref{h11}), (\ref{h01}), (\ref{h02}), (\ref{h03}), (\ref{h04}), (\ref{h05}), (\ref{h44}), (\ref{h112}), and Young's inequality, we derive
\be\nonumber\ba
& \frac{1}{2} \frac{d}{dt} \int \n |\mathbf{u}_{tt}|^2 dx
+ \int \n^\alpha ( |\na \mathbf{u}_{tt}|^2 + (\alpha-1) (\div \mathbf{u}_{tt})^2 ) dx \\
& \le C \int \left( |\n_{tt}| |\mathbf{u}_t| + |\n_t| |\mathbf{u}_{tt}| + |\n_{tt}| |\mathbf{u}| |\na \mathbf{u}| + |\n_t| |\mathbf{u}_t| |\na \mathbf{u}| + |\mathbf{u}_{tt}| |\na \mathbf{u}| + |\mathbf{u}_t| |\na \mathbf{u}_t| \right) |\mathbf{u}_{tt}| dx \\
& \quad + C \int |\n_t| |\mathbf{u}| |\na \mathbf{u}_{t}| |\mathbf{u}_{tt}| dx
+ C \int \left( |\n_{tt}| |\na \mathbf{u}| + |\n_t|^2 |\na \mathbf{u}| + |\n_t| |\na \mathbf{u}_{t}| \right) |\na \mathbf{u}_{tt}| dx \\
& \le C \left( \| \n_{tt} \|_{L^2} \| \mathbf{u}_t \|_{L^\infty}
+ \| \n_t \|_{L^\infty} \| \mathbf{u}_{tt} \|_{L^2}
+ \| \n_{tt} \|_{L^2} \| \mathbf{u} \|_{L^\infty} \| \na \mathbf{u} \|_{L^\infty} \right) \| \mathbf{u}_{tt} \|_{L^2} \\
& \quad + C \left( \| \n_t \|_{L^4} \| \mathbf{u}_t \|_{L^4} \| \na \mathbf{u} \|_{L^\infty}
+ \| \mathbf{u}_{tt} \|_{L^2} \| \na \mathbf{u} \|_{L^\infty}
+ \| \mathbf{u}_t \|_{L^\infty} \| \na \mathbf{u}_t \|_{L^2} \right) \| \mathbf{u}_{tt} \|_{L^2} \\
& \quad + C \| \n_t \|_{L^\infty} \| \mathbf{u} \|_{L^\infty} \| \na \mathbf{u}_t \|_{L^2} \| \mathbf{u}_{tt} \|_{L^2}
+ C \| \n_{tt} \|_{L^2} \| \na \mathbf{u} \|_{L^\infty} \| \na \mathbf{u}_{tt} \|_{L^2} \\
& \quad + C \left( \| \n_t \|_{L^4}^2 \| \na \mathbf{u} \|_{L^\infty}
+ \| \n_t \|_{L^\infty} \| \na \mathbf{u}_t \|_{L^2} \right) \| \na \mathbf{u}_{tt} \|_{L^2} \\
& \le C + C \| \mathbf{u}_t \|_{L^\infty} \| \mathbf{u}_{tt} \|_{L^2}
+ C \| \mathbf{u}_{tt} \|_{L^2}^2 + C \| \na \mathbf{u}_{tt} \|_{L^2} \\
& \le C + C \left( 1 + \| \mathbf{u}_{tt} \|_{L^2} \right) \| \mathbf{u}_{tt} \|_{L^2}
+ C \| \mathbf{u}_{tt} \|_{L^2}^2
+ \frac{1}{2} \int \n^\alpha ( |\na \mathbf{u}_{tt}|^2 + (\alpha-1) (\div \mathbf{u}_{tt})^2 ) dx \\
& \le C \left( 1 + \| \mathbf{u}_{tt} \|_{L^2}^2 \right)
+ \frac{1}{2} \int \n^\alpha ( |\na \mathbf{u}_{tt}|^2 + (\alpha-1) (\div \mathbf{u}_{tt})^2 ) dx,
\ea\ee
which yields
\be\la{h62}\ba
\frac{d}{dt} \int \n |\mathbf{u}_{tt}|^2 dx + \int \n^\alpha ( |\na \mathbf{u}_{tt}|^2 + (\alpha-1) (\div \mathbf{u}_{tt})^2 ) dx
\le C \left( 1 + \| \mathbf{u}_{tt} \|_{L^2}^2 \right).
\ea\ee
Multiplying (\ref{h62}) by $t$ leads to
\be\la{h63}\ba
\frac{d}{dt} \left( t \int \n |\mathbf{u}_{tt}|^2 dx \right)
+ t \int \n^\alpha ( |\na \mathbf{u}_{tt}|^2 + (\alpha-1) (\div \mathbf{u}_{tt})^2 ) dx
\le C \left( 1 + \| \mathbf{u}_{tt} \|_{L^2}^2 \right),
\ea\ee
which together with (\ref{h04}) and Gr\"onwall's inequality implies
\be\la{h64}\ba
\sup_{0 \le t \le T} t \| \mathbf{u}_{tt} \|_{L^2}^2
+ \int_0^T t \| \na \mathbf{u}_{tt} \|_{L^2}^2 dt \le C.
\ea\ee
This, combined with (\ref{h04}), (\ref{h44}), (\ref{h52}), and (\ref{h05}), gives (\ref{h06}) and completes the proof of Lemma \ref{hl6}.
\end{proof}

\section{Proofs of Theorems \ref{th1}--\ref{th2}}
In this section, using the a priori estimates established in Sections 3--5, we prove the main results of this paper in this section.

\noindent\textbf{Proof of Theorem \ref{th1}.}
By the local existence result (Lemma \ref{lct}) and Lemma \ref{dcx}, there exists a $T_*>0$ such that the system (\ref{ns})--(\ref{nxxs})
admits a unique local spherically symmetric classical solution $(\n,\mathbf{u})$ on $\mathbb{R}^N \times (0,T_*]$.
We next extend this local classical solution globally in time.

We set
\be\la{pf1}\ba	
T^* \triangleq \sup \{T \, \big|\, (\n,\mathbf{u}) \text{ is a classical solution on } \mathbb{R}^N \times (0,T] \}.
\ea\ee
For any $0<\tau<T\leq T^*$ with $T$ finite, if follows from (\ref{h01}), (\ref{h02}), (\ref{h03}), (\ref{h04}), (\ref{h05}), and (\ref{h06}) that
\be\la{pf2}\ba
\mathbf{u} \in C \left( [\tau ,T]; H^3(\mathbb{R}^N) \right),
\ea\ee
where we have used the standard embedding
\be\nonumber\ba
L^\infty(\tau,T;H^4(\mathbb{R}^N)) \cap H^1(\tau,T;H^2(\mathbb{R}^N)) \hookrightarrow
C\left( [\tau,T]; H^3(\mathbb{R}^N) \right).
\ea\ee
Furthermore, combining $(\ref{ns})_1$ with (\ref{h01}), (\ref{h02}), (\ref{h03}), (\ref{h04}), (\ref{h05}), and (\ref{h06}) and applying the standard argument in \cite[Lemma 2.3]{L1}, we obtain
\be\la{pf3}\ba
\n \in C \left([0,T];H^3(\mathbb{R}^N) \right).
\ea\ee

Finally, we claim that
\be\la{pf4}\ba
T^*=\infty.
\ea\ee
Otherwise, $T^*<\infty$.
From (\ref{pf2}) and (\ref{pf3}), we conclude that
\be\nonumber\ba
(\n,\mathbf{u})(x,T^*) = \lim_{t \to T^*}(\n,\mathbf{u})(x,t)
\ea\ee
satisfies (\ref{th13}).
Hence, we may take $(\n,\mathbf{u})(x,T^*)$ as the initial data, and Lemma \ref{lct} implies that we can extend the local classical solution beyond $T^*$.
This contradicts (\ref{pf1}), thereby proving (\ref{pf4}).
The uniqueness of $(\n,\mathbf{u})$ satisfying (\ref{th14}) follows from arguments similar to those in \cite{LPZ}.
This completes the proof of Theorem \ref{th1}.

\noindent\textbf{Proof of Theorem \ref{th212}.}
Based on Lemmas \ref{2l1}--\ref{2l4}, \ref{2l7}, and \ref{hl1}--\ref{hl6}, the proof of Theorem \ref{th212} follows an argument similar to that of Theorem \ref{th1} and is thus omitted.

\noindent\textbf{Proof of Theorem \ref{th2}.}
Following the arguments in Sections 4 and 5, we can derive the corresponding global a priori estimates for the three-dimensional bounded domain under the condition (\ref{th21}).
Since the remainder of the proof is analogous to that of Theorem \ref{th1}, it is omitted for brevity.

\bigskip

\noindent\textbf{Data availability.} No data was used for the research described in the article.

\bigskip

\noindent\textbf{Conflict of interest.} The author declares no conflict of interest.

\end{document}